\documentclass[10pt,twoside,reqno]{amsart}
\usepackage[all]{xy}
\usepackage{amssymb, latexsym, amsthm, amsmath, mathtools, mathrsfs}
 \usepackage{enumitem,color}
 
\usepackage{mathtools}

\usepackage{biblatex}
\usepackage{array,longtable}

\usepackage{booktabs}
\usepackage{bigstrut}
\usepackage{hyperref}
\usepackage[capitalize,nameinlink]{cleveref} 
\usepackage{xcolor} 
\hypersetup{
    colorlinks,
    linkcolor={red!80!black},
    citecolor={green!80!black},
    urlcolor={blue!80!black}
}

\long\def\symbolfootnote[#1]#2{\begingroup%
\def\thefootnote{\fnsymbol{footnote}}\footnote[#1]{#2}\endgroup}

\newcommand{\F}{\mathscr{F}}
\newcommand{\A}{\mathscr{A}}

\newcommand{\diag}{\textup{diag}}
\newcommand{\GL}{\mathrm{GL}}

\newcommand{\M}{\mathrm{M}}

\makeatletter
\def\imod#1{\allowbreak\mkern10mu({\operator@font mod}\,\,#1)}
\makeatother

\DeclareUnicodeCharacter{0306}{\v}
\newtheorem{mainthm}{Theorem}

\newtheorem{theorem}{Theorem}[section]
\newtheorem{problem}{Conjecture}
\newtheorem{lemma}[theorem]{Lemma}

\newtheorem{proposition}[theorem]{Proposition}
\newtheorem*{theorem*}{Theorem}
\theoremstyle{definition}

\newtheorem{example}[theorem]{Example}
\numberwithin{equation}{section}
\counterwithin{table}{section}

\newcommand{\ignore}[1]{}

\newcommand{\mynote}[1]{}
\begin{document}
\setcounter{section}{0}
\title{Images of Polynomial maps with constants}
\author[Panja S.]{Saikat Panja}
\address{Harish-Chandra Research Institute Prayagraj (Allahabad), Uttar Pradesh 211019, India}
\email{panjasaikat300@gmail.com}
\author[Saini P.]{Prachi Saini}
\address{IISER Pune, Dr. Homi Bhabha Road, Pashan, Pune 411 008, India}
\email{prachi2608saini@gmail.com}
\author[Singh A.]{Anupam Singh}
\address{IISER Pune, Dr. Homi Bhabha Road, Pashan, Pune 411 008, India}
\email{anupamk18@gmail.com}
\thanks{The first-named author is supported by the HRI PDF-M scholarship. The second-named author acknowledges the support of CSIR PhD scholarship number 09/0936(1237)/2021-EMR-I. The third-named author is funded by an NBHM research grant 02011/23/2023/NBHM(RP)/RDII/5955 for this research.}
\today


\begin{abstract}
Let $K$ be an algebraically closed field and $\M(2,K)$ be the $2\times 2$ matrix algebra over $K$ and $\GL(2,K)$ be the invertible elements in $\M(2,K)$. We explore the image of polynomials with constants, namely from the free algebra $\M(2,K)\langle x, y\rangle$. In this article, we compute the images of the polynomial maps given by 
(a) generalized sum of powers $Ax^{k_1} + By^{k_2}$ and (b) generalized commutator map $Axy -Byx$, 
where $A$, $B$ are non-zero elements of $\M(2,K)$. We compute this in the first case by fixing a simultaneous conjugate pair for $A, B$ and it turns out that it is surjective in most of the cases. In the second case, we show that the image of the map is always a vector space. 
\end{abstract}
\maketitle
\section{Introduction}
\subsection*{Polynomial maps on algebras} 
For a given polynomial, finding a solution in an algebra has been a driving force behind much research since antiquity. 
Be it algebraic geometry, the Diophantine equations or the Waring-type 
problem, the common link among them is finding a solution to a polynomial 
equation, equivalently 
finding the image of the given polynomial. For a $K$-algebra $\A$ and an element $w\in \F_n=K\langle x_1, \ldots, x_n\rangle$, the 
free algebra of rank $n$, we get a natural map
\begin{align*}
    \widetilde{w}\colon\A^n\longrightarrow \A,
\end{align*}
by substitution. These are known as \emph{polynomial maps}.
The motivation behind studying the polynomial maps in algebra comes from the famous Kaplansky-L\'{v}ov conjecture {(see \cite{Kap1957})} (which asserts that the image of a multilinear map on the full matrix algebra
over a field is a vector space) and the expectation of 
results in simple algebra that are parallel to the results in simple groups. The celebrated result of
Larsen, Shalev and Tiep {\color{black}(see \cite{LarShaTie2011})} states that for a non-trivial $w \in \F_d$, the free group on
$d$ generators, there exists a constant $N = N_w$ such that
for all finite non-abelian simple groups $G$ of order greater than $N$, we have
$w(G)^2 = G$. Larsen suggested an analogous result for $\M(n,q)$. It was proved that for sufficiently 
large $q$, every matrix in $\M(n,q)$ can be written as a sum of two $K$-th powers, in
\cite{Kishore2022} and \cite{KishoreSingh2022}. This is equivalent to stating that
the map $x^k+y^k$ is surjective on $\M(n,q)$ for sufficiently large $q$.
This was further generalized in \cite{panja2023surjectivity}, for the map $\sum\limits_{i=1}^{m}
\alpha_i x_i^{k_i}$, for $m\geq 2$ and the rings $\M(n,D)$ with $D=\mathbb{C}, \mathbb{F}_q, \mathbb{R},
\mathbb{H}_\mathbb{R}$. This is commonly known as \emph{Waring-type problem} which asks if
the elements of the subalgebra generated by the image $\widetilde{\omega}(\A^m)$ can be written as a sum of $\ell$ many elements of
$\widetilde{\omega}(\A^m)$, where $\ell$ is a fixed positive integer. Let us list out some references where other algebras have been considered, the case 
central simple algebras in \cite{Pu07}, the case of $\M(n, R)$ with $R$ being a commutative ring in \cite{KaGa13}, the case of $\M(n,\mathbb Z)$ in \cite{Lee2018}, the upper triangular matrix algebra in 
\cite{kaushik2023waring}.
Other instances of Waring-type problems occur in the works \cite{Bresar2020}, \cite{BresarSemerl2022} and \cite{BresarSemerl2021}.
The Kaplansky-L\'{v}ov conjecture has been quite a driving force to a tremendous amount of research in recent times. While it  
remains open for most of the cases, some progress has been made for the case of $\M(2,K)$ in 
\cite{BeMaRo12}, the case of $\M(3,K)$ in \cite{BeMaRo16}. We remark that the last two articles do not 
solve the problem for a general field, but do it with mild restrictions on the field. A parallel 
result holds in the case of upper triangular matrix algebras, which is known as {F}agundes-{M}ello 
conjecture, and was solved in \cite{GadeMe22} an \cite{LuWa22} independently. A more generalized 
result has been proved in \cite{PaPr23}. For a survey, we suggest the reader look into \cite{BeMaRoYa20} and the references therein.

\subsection*{Polynomial maps with constants on algebras}

Let $K$ be a field, $\F_n$ denote the free algebra of rank $n$ over $K$ and $\A$ be an associative algebra over $K$. An element $\omega$ of the free associative algebra $\A*\F_n = \A\langle x_1, \ldots, x_n\rangle$ induces a map on $\A$
\begin{align*}
    \widetilde{\omega}\colon\A^n\longrightarrow\A,
\end{align*}
through evaluation. These maps are called \emph{polynomial maps with constants}. These are a natural generalization of polynomial maps. 
Before moving further we present a few examples.
\begin{example}
Let $K$ be a field and $\A = \M(n,k)$ be the matrix algebra. For positive integers $k_1, \ldots, k_m$ we consider $\omega = A_1 x_1^{k_1} + \cdots +A_m x_m^{k_m} \in \M(n,k)\langle x_1, \ldots, x_m\rangle$ for $A_1, \ldots A_m \in \M(n,k)$. This could be called a generalised diagonal map.  
\end{example}
\begin{example}
Let $K$ be a field and $\A = \M(n,k)$ be the matrix algebra. Consider $\omega = Axy - Byx \in \M(n,k)\langle x, y\rangle$ for $A, B \in \M(n,k)$.  
\end{example}
There is a parallel study of these maps, in the case of groups, which are known as \emph{word maps with constants}, see \cite{GoNiKu18}. The goal of this paper is to broaden the viewpoint on the subjectivity problem from polynomial maps on algebras to polynomial maps with constants on algebras. In this paper, we study the maps in the above examples over $\M(2,k)$ with two variables and determine the images. We prove the following:  

\begin{mainthm}\label{thm:A}
Let $K$ be an algebraically closed field. Consider the polynomial map 
$\widetilde{w}$ given by $Ax^{k_1} + By^{k_2}$ on $M_2(k)$ where $A, B \in 
\M(2,K)$ both non-zero. Then,
$\widetilde{w}$ is surjective if and only if $A$ and $B$ can be simultaneously 
conjugated to a pair of matrices such that both the matrices do not have the 
same zero rows. Further, for a choice of $A, B$ up to simultaneous conjugation 
the images are described in \cref{tableIV}, 
which are vector spaces of $\M(2,K)$. 
\end{mainthm}
\noindent The above theorem can be seen as an analogue of the Matrix-Waring problem.

\begin{mainthm}\label{thm:B}
Let $K$ be an algebraically closed field. Consider the polynomial map $\widetilde{w}$ given by $Axy - Byx$ on $M_2(k)$ where $A, B \in \M(2,K)$ both non-zero. Then, the image is a vector space.
\end{mainthm}
\noindent This theorem is parallel to that of the L\'{v}ov-Kaplansky 
conjecture. 
\begin{longtable}{l|c|r}
\toprule
Choice of $A$& Choice of $B$& Image of $\widetilde{w}$ \bigstrut\\
    \midrule\endfirsthead
    \toprule
Choice of $A$& Choice of $B$& Image of $\widetilde{w}$ \bigstrut\\
    \midrule\endhead
$\diag(\lambda, \lambda)$ & $\diag(\xi, \xi)$ & $\M(2,K)$  \\
$\diag(\lambda, \lambda)$ & $\diag(\xi_1, \xi_2), \xi_1\neq\xi_2$  & $\M(2,K)$  \\
$\diag(\lambda, \lambda)$ & $\begin{pmatrix} \xi & 1 \\ & \xi\end{pmatrix}$ & $\M(2,K)$  \\
$\diag(\lambda, \mu), \lambda\neq \mu$ & $\diag(\xi, \xi)$ & $\M(2,K)$  \\
$\diag(\lambda, \mu), \lambda\neq \mu$ & $\diag(\xi_1, \xi_2), \xi_1\xi_2\neq 0$ & $\M(2,K)$  \\
$\diag(\lambda, \mu), \lambda\neq 0$ & $\diag(\xi_1, \xi_2), \xi_2\neq 0$ & $\M(2,K)$  \\
$\diag(\lambda, \mu), \mu\neq 0$ & $\diag(\xi_1, \xi_2), \xi_1\neq 0$ & $\M(2,K)$  \\
$\diag(\lambda, 0)$ & $\diag(\xi, 0)$ & $\begin{pmatrix} * & * \\ & \end{pmatrix}$  \\
$\diag(0, \mu)$ & $\diag(0, \xi)$ & $\begin{pmatrix}  &  \\ *&* \end{pmatrix}$  \\
$\diag(\lambda, \mu), \lambda\mu\neq 0$ &  $\begin{pmatrix} \xi & 1  \\ &\xi \end{pmatrix}$  & $\M(2,K)$\\
$\diag(\lambda, \mu), \lambda\mu = 0$ &  $\begin{pmatrix} \xi & 1  \\ &\xi \end{pmatrix}, \xi\neq 0$  & $\M(2,K)$\\
$\diag(\lambda, \mu), \mu \neq 0$ &  $\begin{pmatrix} 0 & 1  \\ &0 \end{pmatrix}$  & $\M(2,K)$\\
$\diag(\lambda, 0)$ &  $\begin{pmatrix} 0 & 1  \\ &0 \end{pmatrix}$  &$\begin{pmatrix} * & * \\ & \end{pmatrix}$  \\
$\diag(\lambda, \mu), \lambda\mu\neq 0$ &  $\begin{pmatrix} \xi_1 & 1  \\ &\xi_2 \end{pmatrix}, \xi_1\neq \xi_2$  & $\M(2,K)$\\
$\diag(\lambda, \mu), \lambda\mu = 0$ &  $\begin{pmatrix} \xi_1 & 1  \\ &\xi_2 \end{pmatrix}, \xi_1\xi_2\neq 0$  & $\M(2,K)$\\
$\diag(\lambda, \mu), \mu \neq 0$ &  $\begin{pmatrix} \xi & 1  \\ &0 \end{pmatrix}$  & $\M(2,K)$\\
$\diag(\lambda, 0)$ &  $\begin{pmatrix} \xi & 1  \\ &0 \end{pmatrix}$  &$\begin{pmatrix} * & * \\ & \end{pmatrix}$  \\
$\diag(\lambda, 0)$ &  $\begin{pmatrix} 0& 1  \\ &\xi \end{pmatrix}, \xi\neq 0$  & $\M(2,K)$  \\
$\diag(0, \mu)$ &  $\begin{pmatrix} 0 & 1  \\ &\xi \end{pmatrix}$  &$\M(2,K)$  \\
$\diag(\lambda, \mu), \lambda\mu\neq 0$ &  $\begin{pmatrix} \xi &   \\ 1&\xi \end{pmatrix}$  & $\M(2,K)$\\
$\diag(\lambda, \mu), \lambda\mu = 0$ &  $\begin{pmatrix} \xi &   \\ 1&\xi \end{pmatrix}, \xi\neq 0$  & $\M(2,K)$\\
$\diag(\lambda, \mu), \lambda \neq 0$ &  $\begin{pmatrix} 0 & 0  \\ 1 &0 \end{pmatrix}$  & $\M(2,K)$\\
$\diag(0, \mu)$ &  $\begin{pmatrix} 0 & 0  \\ 1&0 \end{pmatrix}$  &$\begin{pmatrix} 0 & 0 \\ *&* \end{pmatrix}$  \\ 
$\diag(\lambda, \mu), \lambda\mu\neq 0$ &  $\begin{pmatrix} \xi_1 & 0  \\ 1&\xi_2 \end{pmatrix}, \xi_1\neq \xi_2$  & $\M(2,K)$\\
$\diag(\lambda, \mu), \lambda\mu = 0$ &  $\begin{pmatrix} \xi_1 & 0  \\ 1&\xi_2 \end{pmatrix}, \xi_1\xi_2\neq 0$  & $\M(2,K)$\\
$\diag(\lambda, 0)$ &  $\begin{pmatrix} \xi &   \\ 1 &\xi_2 \end{pmatrix}$  & $\M(2,K)$  \\
$\diag(0, \mu)$ &  $\begin{pmatrix}  \xi_1& 0  \\ 1&\xi_2 \end{pmatrix}, \xi_1\neq 0$  &$\M(2,K)$  \\ 
$\diag(0, \mu)$ &  $\begin{pmatrix}  0& 0  \\ 1 &\xi \end{pmatrix}$  &  $\begin{pmatrix} 0 & 0 \\ *&* \end{pmatrix}$ \\ 
$\diag(\lambda, \mu), \lambda \neq \mu$ &  $\begin{pmatrix} \xi_1 & \xi_2  \\ 1&\xi_3 \end{pmatrix}, \xi_i\neq 0$  & $\M(2,K)$\\
$\diag(\lambda, \mu), \lambda\mu\neq 0$ &  $\begin{pmatrix} 0&\xi_2    \\ 1 &\xi_3 \end{pmatrix}, \xi_3\neq 0$  & $\M(2,K)$  \\
$\diag(\lambda, \mu)$ &  $\begin{pmatrix}  0&\xi_2  \\ 1&\xi_3 \end{pmatrix}, \xi_i\neq 0$  &$\M(2,K)$  \\ 
$\diag(0, \mu)$ &  $\begin{pmatrix}  0& 0  \\ 1 &\xi \end{pmatrix}$  &  $\begin{pmatrix} 0 & 0 \\ *&* \end{pmatrix}$ \\ 
$\diag(\lambda, \mu), \lambda \neq \mu$ &  $\begin{pmatrix} \xi_1 & \xi_2  \\ 1& 0 \end{pmatrix}, \xi_1\neq 0$  & $\M(2,K)$\\
$\diag(\lambda, \mu), \lambda\neq \mu$ &  $\begin{pmatrix} 0&\xi    \\ 1 & 0 \end{pmatrix}, \xi\neq 0$  & $\M(2,K)$  \\
$\begin{pmatrix}\lambda & 1\\ & \lambda\end{pmatrix}, \lambda\neq 0$ &  $\begin{pmatrix} \xi &   \\ z &\xi \end{pmatrix}, \xi z\neq 0$  &$\M(2,K)$  \\
$\begin{pmatrix} & 1\\ 0&0\end{pmatrix}$ &  $\begin{pmatrix} \xi_1 &   \\ z &\xi_2 \end{pmatrix}, \xi_1\neq \xi_2, z\neq 0$  &$\M(2,K)$  \\
$\begin{pmatrix}\lambda & 1\\ & \lambda\end{pmatrix}, \lambda\neq 0$ &  $\begin{pmatrix} \xi_1 &   \\  &\xi_2 \end{pmatrix}, \xi_1 \neq \xi_2$  &$\M(2,K)$  \\
$\begin{pmatrix} 0 & 1\\ & 0 \end{pmatrix}$ &  $\begin{pmatrix}  & z  \\ 0 &0 \end{pmatrix}$  &   $\begin{pmatrix} * & * \\ & \end{pmatrix}$ \\
$\begin{pmatrix}\lambda & 1\\ & \lambda\end{pmatrix}$ &  $\begin{pmatrix} \xi & z  \\  &\xi \end{pmatrix}, \xi \neq 0$  &$\M(2,K)$  \\
\caption{Images of $Ax^{k_1} + By^{k_2}$}\label{tableIV}
\end{longtable}
\vskip2mm

Hereafter, \textcolor{blue}{it will be assumed that the field $K$ 
is an algebraically closed field}.
We deploy the method of simultaneous conjugation to simplify the problem.
In general finding the class representative for $(A, B)$ under simultaneous conjugacy is a wild problem. However, for $n=2$ we write this explicitly in 
\cref{sec:reduction}. With the help of this, we compute the images for each of 
those cases and prove \cref{thm:A} in \cref{sec:proof-theorem-A}. The 
\cref{thm:B} is proved \cref{sec:proof-theorem-B}. We hope this work sheds some light on the general problem.


\subsection*{Acknowledgement} This work started when the first-named author visited IISER Pune during a workshop in January 2023. The work was completed during another workshop at HRI in December 2023 organised by Professor Manoj K. Yadav. We take this opportunity to thank the organizers and the hospitality of the respective institutes.

\section{Representatives for the simultaneous conjugacy classes}\label{sec:reduction}
In a previous work \cite{panja2023surjectivity}, the authors have proved that
the map induced by $\alpha x^{k_1}+\beta y^{k_2}$ for $\alpha\beta\neq 0$, 
is surjective on 
$\M(n,\mathbb{C})$, and it also follows from their proof that
the result holds for any algebraically closed field. This motivates us to look
into the case of the polynomial $Ax^{k_1}+By^{k_2}$. Now we explain our
strategy to tackle the mentioned problem in \cref{sec:proof-theorem-A}. A
somewhat similar method will be further used in \cref{sec:proof-theorem-B} and we will keep referring to this section, whenever needed.
For each $A\in \M(2,K)$, there exists $P_A \in \GL(2,K)$ such that $P_AAP_A^{-1}=J_A$ where $J_A$ denote the Jordan canonical form of $A$. We want to comprehend for which matrices in $\M(2,K)$ the solution of equation $Ax^{k_1}+By^{k_2}$ exists i.e. for which $C$ there exists matrices $X$ and $Y$ in $\M(2,K)$ such that $C=AX^{k_1}+BY^{k_2}$.
On conjugating by $P_A$, we get $C_A=J_Ax^{k_1}+B_Ay^{k_2}$ where $B_A$ and $C_A$ represent matrices corresponding to $B$ and $C$ obtained under conjugation by $P_A$ respectively. 

Let $C_{\GL_2}\left(J_A\right)$ denote the centralizer of $J_A$ in $\GL(2,K)$. Consider the group action of $C_{\GL_2}\left(J_A\right)$ on $\M(2,K)$ given as
\begin{align*}
    C_{GL_2}\left(J_A\right) \times \M(2,K) \longrightarrow \M(2,K) \\
(T, B_A) \longmapsto TB_AT^{-1}
\end{align*}
We consider the orbit space under the group action and determine the image i.e. for each $T\in C_{GL_2}\left(J_A\right)$, it is sufficient to determine the solution of $J_Ax^{k_1}+TB_AT^{-1}y^{k_2}$ and instead of taking $TB_AT^{-1}$ independently we consider the representative of each type. Thus we determine the orbit space under the action of each centralizer of each type of Jordan form existing over $K$. The Jordan forms appearing for $A \in \M(2,K)$ are one of the following type$$ \begin{pmatrix}
    \lambda & 0 \\
    0 & \lambda 
\end{pmatrix}_{\lambda\in k^{\times}} \begin{pmatrix}
    \lambda & 0 \\
    0 & \mu 
\end{pmatrix}_{\lambda\neq \mu} \begin{pmatrix}
    \lambda & 1 \\
    0 & \lambda 
\end{pmatrix}_{\lambda\in k}. $$ 
The \cref{Table-Centralizer} describes the centralizer corresponding to each Jordan form.

\begin{longtable}{c|c}
      \hline
   $J_A $  & $C_{\GL_2}\left(J_A\right)$ \bigstrut\\ 
   \midrule\endfirsthead
    \toprule
$J_A $  & $C_{\GL_2}\left(J_A\right)$ \bigstrut\\
    \midrule\endhead
  \\ $\begin{pmatrix}
    \lambda & 0 \\
    0 & \lambda 
\end{pmatrix}_{\lambda \neq 0}$ & $\GL(2,K)$  \\ [0.75 cm]
\hline
  \\
$\begin{pmatrix}
    \lambda & 0 \\
    0 & \mu 
\end{pmatrix}$ & $\left\{\begin{pmatrix}
    d_1 & 0 \\
    0 & d_2
\end{pmatrix}\middle| d_1d_2\neq 0\right\}$ \\ [0.75 cm]
\hline
  \\
$\begin{pmatrix}
    \lambda & 1 \\
    0 & \lambda 
\end{pmatrix}$ & $\left\{\begin{pmatrix}
    a_1 & b_1 \\
    0 & a_1
\end{pmatrix}\middle| a_1\neq 0\right\}$\\
\caption{Description of the centralizers}\label{Table-Centralizer}
\end{longtable}    
Now we aim to determine the orbit space under the action of each centralizer mentioned in \cref{Table-Centralizer} and note down the representative of the orbit spaces under the action. Let $B_A=\begin{pmatrix}
    a' & b'\\
    c' & d'
\end{pmatrix} \in \M(2,K).$

 \subsection{The case $C_{\GL_2} = \GL(2,K)$.} Then considering $B_A$ upto conjugation the orbit space have representatives as : $$ \begin{pmatrix}
    \mu_1 & 0 \\
    0 & \mu_1 
\end{pmatrix}_{\mu\in k^{\times}}, \begin{pmatrix}
    \mu_1 & 0 \\
    0 & \mu_2 
\end{pmatrix}_{\mu_1\neq \mu_2}, \begin{pmatrix}
    \mu_1 & 1 \\
    0 & \mu_1 
\end{pmatrix}_{\mu_1\in k}. $$ 
\subsection{The case $C_{\GL_2} = \left\{\begin{pmatrix}
    d_1 & 0 \\
    0 & d_2
\end{pmatrix}\middle| d_1d_2\neq 0\right\}$.}
Let $T=\begin{pmatrix}
    d_1 & 0\\
    0 & d_2
\end{pmatrix}$ in $C_{\GL_2}$, consider $TB_AT^{-1}$ given by $\begin{pmatrix}
    a' & d_1d_2^{-1}b'\\
    d_1^{-1}d_2c' & d'
\end{pmatrix}$. If $c$ is non-zero then it can be scaled to $1$ and if $c$ is zero and $b$ is non-zero, we scale $b$ upto $1$ and hence the representatives are:\\
\begin{gather*}
  \quad \quad \quad \begin{pmatrix}
   \mu_1 & 0 \\
    0 & \mu_1 \end{pmatrix}_{\mu\in k^{\times}}, \begin{pmatrix}
    \mu_1 & 0 \\
    0 & \mu_2 
\end{pmatrix}_{\mu_1\neq \mu_2}, \begin{pmatrix}
    \mu_1 & 1 \\
    0 & \mu_1 
\end{pmatrix}_{\mu_1\in k},\begin{pmatrix}
    \mu_1 & 1 \\
    0 & \mu_2 
\end{pmatrix}_{\mu_1\neq \mu_2},\begin{pmatrix}
    \mu_1 & 0 \\
    1 & \mu_1 
\end{pmatrix}_{\mu_1\in k},\\  \vspace{2cm}
\quad \quad \begin{pmatrix}
    \mu_1 & 0 \\
    1 & \mu_2 
\end{pmatrix}_{\mu_1\neq \mu_2}, \begin{pmatrix}
    z_1 & z_2 \\
    1 & z_3 
\end{pmatrix}_{z_i\neq 0},\begin{pmatrix}
    0 & z_2 \\
    1 & z_3 
\end{pmatrix}_{z_3\neq 0}, \begin{pmatrix}
    z_1 & z_2 \\
    1 & 0 
\end{pmatrix}_{z_1\neq 0}, \begin{pmatrix}
    0 & z_2 \\
    1 & 0 
\end{pmatrix}_{z_2\neq 0}. 
\end{gather*}\vspace{0.25cm}
\subsection{The case $C_{\GL_2}=\left\{\begin{pmatrix}
    a_1 & b_1 \\
    0 & a_1
\end{pmatrix}\middle| a_1\neq 0\right\}$.} 
Then $TB_AT^{-1}$ is $$\begin{pmatrix}
    a'+c'a_1^{-1}b_1 & \left(d'-a'\right)a_1^{-1}b_1+b'-c'(a_1^{-1}b_1)^2\\
    c' & d'-c'a_1^{-1}b_1
\end{pmatrix}$$ where $T\in C_{GL_2}.$ 
\color{black}
Since we can vary the matrices $\begin{pmatrix}
    a_1 & b_1 \\ & a_1
\end{pmatrix}\in C_{\GL_2}$, and $a_1\neq 0$, we denote the variable $a_1^{-1}b_1$ by $x$. Then $TB_AT^{-1}$ is given by
\begin{align*}
    \begin{pmatrix}
        a'+c'x & (d'-a')x+b'-c'x^2\\
        c' & d'-c'x
    \end{pmatrix}.
\end{align*}
Consider the case, when $c'\neq 0$. Then the matrix representatives are given by $\begin{pmatrix}
    \mu _1 & 0\\
    z & \mu_1
\end{pmatrix}_{\mu_1\in k^\times,z\in k}$, and $\begin{pmatrix}
    \mu_1 & 0\\
    z & \mu_2
\end{pmatrix}_{\mu_1\neq\mu_2,z\in k}$. If $c'=0$ and $d'\neq a'$, then the representatives are given by
$\begin{pmatrix}
 \mu_1 & \\ & \mu_2   
\end{pmatrix}_{\mu_1\neq\mu_2\in k}$. Finally if $c'=0$ and $a'=d'$, we have that the elements $\begin{pmatrix}
    a' & b'\\ & a'
\end{pmatrix}$ commute with the elements of $C_{\GL_2}$, hence giving the class representatives to be
$\begin{pmatrix}
    \mu_1&z\\ & \mu_1
\end{pmatrix}$ for $\mu_1,z\in k$, without being a zero matrix.
\color{black}

\section{Proof of theorem A}\label{sec:proof-theorem-A}
Let $\widetilde{B}$ denote the representative of an orbit space. By the action of $C_{\GL_2}$ on $B_A\in 
\M(2,K)$, we have reduced the equation $J_Ax^{k_1}+B_Ay^{k_2}$ to $J_Ax^{k_1}+\widetilde{B}y^{k_2}$. We lay 
emphasis on $C\in \M(2,K)$ which can be written as $J_AX^{k_1}+\widetilde{B}Y^{k_2}$ for some matrices $X$ 
and $Y$ in $\M(2,K)$. Also, if there exists $C'$ such that it can not be written as 
$J_AX^{k_1}+\widetilde{B}Y^{k_2}$ for any matrices $X$ and $Y$ in $\M(2,K)$ then the original map is not 
surjective as $P_A^{-1}T^{-1}C'TP_A$ is not in the image of $Ax^{k_1}+By^{k_2}$ where $P_A$ is the 
matrix corresponding to $A$ such that $P_AAP_A^{-1}=J_A$ and $T\in C_{\GL_2}\left(J_A\right)$. 
If $C$ is non-singular and $J_A$ or $\widetilde{B}$ is non-singular then taking $y=0$ or $x=0$ respectively 
gives us the solution, as $J_A^{-1}C$ or ${\widetilde{B}}^{-1}C$ can be written as $k_1$-th or $k_2$-th 
power over $\M(2,K)$. Indeed, we can do it up to conjugacy and the semisimple case is trivial to solve.
For the non-semisimple case, note that $\begin{pmatrix}
    \lambda & 1\\ & \lambda
\end{pmatrix}^n=\begin{pmatrix}
    \lambda^n& n\lambda^{n-1}\\ & \lambda^n
\end{pmatrix}$, which is conjugate to $\begin{pmatrix}
    \lambda^n & 1\\ & \lambda^n
\end{pmatrix}$, whenever $\lambda\neq 0$.
We contemplate the cases when $C$ is non-singular with $J_A$ and $\widetilde{B}$ are 
singular and the case when $C$ is singular. We divide the proof into three propositions. The first one of them is as follows.
\begin{proposition}\label{lem:power-A-scalar}
    Let $\omega=Ax^{k_1}+ By^{k_2}\in\M(2,K)\langle x, y\rangle$, with $A$, $B$ nonzero matrices. If $A$ is a scalar matrix, then the map $\widetilde{\omega}$ is surjective.
\end{proposition}
\subsection{Proof of \cref{lem:power-A-scalar}}
   We complete the proof in several subcases. They depend on the representative 
   of the pair $(A, B)$ under simultaneous conjugation. We refer to \cref{sec:reduction} for the representatives under simultaneous conjugation.
    The equation we have in hand for $\lambda\neq 0$, is  
$$ C=\begin{pmatrix}
    a & b\\
    c & d
\end{pmatrix}=\begin{pmatrix}
    \lambda & 0\\
    0 & \lambda
\end{pmatrix}x^{k_1}+\widetilde{B}y^{k_2}$$
Considering the representative of each orbit space corresponding to $J_A$ we have the following cases:
\subsubsection{For $\mu_1 \neq 0$, $\widetilde{B}=\begin{pmatrix}\mu_1 & 0\\
        0 & \mu_1
    \end{pmatrix}$.} Then we have  $\begin{pmatrix}
        a & b\\
        c & d
    \end{pmatrix}=\lambda x^{k_1}+\mu_1 y^{k_2}$ for which the given map is surjective by section $5$ of \cite{panja2023surjectivity}.
    \subsubsection{For $\mu_1\neq \mu_2$, $\widetilde{B}=\begin{pmatrix}
        \mu_1 & 0\\
        0 & \mu_2
    \end{pmatrix}$.} There are three possibilities for $\widetilde{B}$ in hand :
    $$\begin{pmatrix}
        \mu_1 & 0\\
        0 & \mu_2
    \end{pmatrix}_{\mu_1\mu_2\neq 0} \text{or} \begin{pmatrix}
        \mu_1 & 0\\
        0 & 0
    \end{pmatrix}_{\mu_1\neq 0} \text{or} \begin{pmatrix}
        0 & 0\\
        0 & \mu_2
    \end{pmatrix}_{\mu_2\neq 0}$$
   It is adequate to consider $C$ to be singular as $A$ is non-singular.

\paragraph{} For $\mu_1\mu_2\neq 0$, 
let $x=\begin{pmatrix}
            a_0 & 0 \\
            0 & a_1
\end{pmatrix}$ where $a_0$ and $a_1$ are elements in $K$ to be fixed later. Then $C-\lambda x^{k_1}=\begin{pmatrix}
a-\lambda a_0^{k_1} & b\\
c & d-\lambda a_1^{k_1}            
\end{pmatrix}$. If $a$ or $d$ are non-zero then $a_0$ and $a_1$ can 
be chosen such that $\det\left(C-\lambda x^{k_1}\right)$ is non-zero 
and hence $\widetilde{B}^{-1}\left(C-\lambda x^{k_1}\right)$ can be 
written as $Y^{k_2}$ for some matrix $Y$ in $\M(2,K)$. Thus if $a$ 
and $d$ both are zero then $C$ being singular implies $b$ or $c$ is 
zero. For $b=0$ or $c=0$, consider $x=I$ where $I$ is an identity 
matrix then $C-\lambda I$ is non-singular and hence there exist $Y$ 
such that $\widetilde{B}^{-1}\left(C-\lambda I\right)$ is $Y^{k_2}$. 
\paragraph{} For $\mu_1\neq 0$ and $\mu_2=0$. If $d\neq 0$, let $y=\begin{pmatrix}
    b_0 & 0\\
    0 & 0
\end{pmatrix}$ and if $d=0$ with $c\neq 0$, consider $y= \begin{pmatrix}
    0 & b_0\\
    0 & 1
\end{pmatrix}$. Then $C-\widetilde{B}y^{k_2}$ is either $\begin{pmatrix}
    a-\mu_1 b_0^{k_2} & b\\
    c & d
\end{pmatrix}$ or $\begin{pmatrix}
    a & b-\mu_2 b_0\\
    c & d
\end{pmatrix}$. In each case $b_0$ can be chosen such tha$C-\widetilde{B}y^{k_2}$ is non-singular and hence can be written a$\lambda X^{k_1}$ for some matrix $X$ in $\M(2,K)$. For $d$ and $c$ both being zero, choose $x=\begin{pmatrix}
    a_0 & 0\\
    0 & 0
\end{pmatrix}$ and $y= \begin{pmatrix}
    0 & b_0\\
    0 & 1
\end{pmatrix}$. Then $J_Ax^{k_1}+\widetilde{B}y^{k_2}=\begin{pmatrix}
    \lambda a_0^{k_1} & \mu_1 b_0\\
    0 & 0
\end{pmatrix}$ where $a_0$ and $b_0$ can be chosen such that $a=\lambda_0^{k_1}$ and $b=\mu_1 b_0$ and hence giving the solution. 
A similar argument for the case $\mu_1=0$, $\mu_2\neq =0$, proves that the map is {\color{red} surjective} in this case as well.
\subsubsection{For $\mu_1 \in k$, $\widetilde{B}=\begin{pmatrix}
           \mu_1 & 1\\
           0 & \mu_1
       \end{pmatrix}$.} There are two instances either $\mu_1=0$ or $\mu_1\neq 0$.
       
        For $\mu_1 =0$, let $y= \begin{pmatrix}
               1 & 0\\
               b_0 & 0
           \end{pmatrix}$ if $d\neq 0$ and $y=\begin{pmatrix}
               0 & 0\\
               0 & b_0
           \end{pmatrix}$ for $d=0$ but $c\neq 0$. Then in each case $C-\widetilde{B}y^{k_2}$ is non-singular and hence is in image of $\lambda x^{k_1}$. If $c$ and $d$ both are zero then take $x=\begin{pmatrix}
               a_0 & a_1\\
               0 & 0
           \end{pmatrix}$ and $y=0$ for $a\neq 0$ otherwise take $x=0$ and $y=\begin{pmatrix}
               0 & 0 \\
               0 & b_0
           \end{pmatrix}$. If $a\neq 0$, then $C=\lambda x^{k_1}+\widetilde{B}0$ gives two equations $a=\lambda a_0^{k_1}$ and $b=\lambda a_1 a_0^{k-1}$. Since $K$ is algebraically closed, $a_0$ and $a_1$ can be chosen so as to have the solution of equations. Suppose $a=0$, then $\lambda x^{k_1}+\widetilde{B}y^{k_2}=\begin{pmatrix}
               0 & b_0^{k_2}\\
               0 & 0
           \end{pmatrix}$. Choose $b_0$ such that $b=b_0^{k_2}$ and hence $C$ is in the image proving that the map is {\color{red} surjective}. 
       {For $\mu_1\neq 0$} The proof follows from case $2$ part $a$.
       This completes the case when the Jordan form of $A$ is given by a scalar matrix. \qed

\begin{proposition}\label{lem:power-A-diagonal}
    Let $\omega=Ax^{k_1}+ By^{k_2}\in\M(2,K)\langle x, y \rangle$, with $A$, $B$ nonzero 
    matrices. If up to conjugation $A$ is a diagonal matrix $\diag(\lambda, 
    \mu)$, then the map $\widetilde{\omega}$ is surjective if and only if exactly one of the following happens
    \begin{enumerate}
        \item $A$ is invertible
        \item $\lambda=0$ and the second row of an orbit representative of $B$ is nonzero,
        \item $\mu=0$ and the first row of an orbit representative of $B$ is nonzero.
    \end{enumerate}
\end{proposition}
\subsection{Proof of \cref{lem:power-A-diagonal}}
In the current section, we get a handle on the Jordan form being diagonal with distinct diagonal entries. The equation being dealt here for $\lambda\neq \mu$ is  given by $$C=\begin{pmatrix}
    a & b\\
    c & d
\end{pmatrix}=\begin{pmatrix}
    \lambda & 0\\
    0 & \mu
\end{pmatrix}x^{k_1}+\widetilde{B}y^{k_2}.$$
Considering the representative of each orbit space obtained by the action of the centralizer of the given Jordan form, we have the following cases:
\subsubsection{For $\mu_1\neq 0$, $\widetilde{B}= \begin{pmatrix}
        \mu_1 & 0\\
        0 & \mu_1
    \end{pmatrix}$} The given map is surjective which follows by Case $2$ of Section $3.1$. 
\subsubsection{For $\mu_1\neq \mu_2$, $\widetilde{B}=\begin{pmatrix}
        \mu_1 & 0\\
        0 & \mu_2
    \end{pmatrix}$} For $C$ being singular and both $J_A$ and $\widetilde{B}$ are non-singular, we choose $\zeta\in k^{\times}$ such that $\zeta$ is not a characteristic value of $C$. Let $x=\begin{pmatrix}
        \sqrt[k_1]{\lambda^{-1}\zeta} & 0\\
        0 & \sqrt[k_1]{\mu^{-1}\zeta}
    \end{pmatrix}$. Then $C-J_Ax^{k_1}=C-\zeta I$ is a non-singular matrix and hence $y$ can be chosen such that $C=J_Ax^{k_1}+\widetilde{B}y^{k_2}$. For $J_A$ being non-singular or $\widetilde{B}$ is non-singular, the proof is similar for both. Without loss of generality, let us assume $J_A$ is non-singular and $\widetilde{B}$ is singular having two choices either $\mu_1=0$ or $\mu_2=0$. For $\mu_1=0$, let $y=\begin{pmatrix}
        0 & 0\\
        0 & b_0
    \end{pmatrix}$. Choose $b_0\in k$ such that $C-\widetilde{B}y^{k_2}$ is non-singular with $a$ being non-zero and hence lying in image of $J_Ax^{k_1}$. If $a$ is zero then $C$ being singular gives us $bc=0$ and hence if $b=0$, let $y=\begin{pmatrix}
        0 & 0\\
        b_0 & 1
    \end{pmatrix}$ and $x=\begin{pmatrix}
        0 & 0\\
        0 & a_0
    \end{pmatrix}$ then $J_Ax^{k_1}+\widetilde{B}y^{k_2}=\begin{pmatrix}
        0 & 0\\
        \mu_2 b_0 & \mu a_0^{k_1}+\mu_2 
    \end{pmatrix}$. We can choose $a_0$ such that $d=\mu a_0^{k_1}+\mu_2$ has a solution over $K$ and similarly $b_0$ can be chosen such that $c=\mu_2 b_0$. Now if $c=0$, then let $y= \begin{pmatrix}
        0 & 0\\
        0 & b_0
    \end{pmatrix}$ and $x=\begin{pmatrix}
        0 & a_0\\
        0 & 1
    \end{pmatrix}$. Then $J_Ax^{k_1}+\widetilde{B}y^{k_2}=\begin{pmatrix}
        0 & \lambda a_0\\
        0 & \mu+\mu_2 b_0^{k_2}
    \end{pmatrix}$ and hence choosing $a_0$ and $b_0$ such that $b=\lambda a_0$ and $d=\mu+\mu_2 b_0^{k_2}$ have solutions over $K$. Similarly, if $\mu_2=0$, for $d$ being non-zero, choose $y=\begin{pmatrix}
        b_0 & 0\\
        0 & 0
    \end{pmatrix}$ such that $C-\widetilde{B}y^{k_2}$ given by the matrix $\begin{pmatrix}
        a-\mu_1 b_0^{k_2} & b\\
        c & d
    \end{pmatrix}$ is non-singular and hence is in the image of $J_Ax^{k_1}$. If $d$ is zero, then $bc=0$. If $d$ and $b$ both are zero then let $x=\begin{pmatrix}
        1 & 0\\
        a_0 & 0
    \end{pmatrix}$ and $y=\begin{pmatrix}
        b_0 & 0\\
        0 & 0
    \end{pmatrix}$ such that $J_Ax^{k_1}+\widetilde{B}y^{k_2}=\begin{pmatrix}
        \lambda +\mu_1 b_0^{k_2} & 0\\
        \mu a_0 & 0
    \end{pmatrix}$. Choose $a_0$ and $b_0$ such that $a=\lambda +\mu_1 b_0^{k_2}$ and $c=\mu a_0$. If $d$ and $c$ are zero then $x=\begin{pmatrix}
        a_0 & 0\\
        0 & 0 
    \end{pmatrix}$ and $y=\begin{pmatrix}
        1 & b_0\\
        0 & 0
    \end{pmatrix}$. Then $J_Ax^{k_1}+\widetilde{B}y^{k_2}$ is $\begin{pmatrix}
        \lambda a_0^{k_1}+\mu_1 & \mu_1 b_0\\
        0 & 0
    \end{pmatrix}$ and choosing $a_0$ and $b_0$ accordingly gives us the desired result. 
    
    The only case left is when $J_A$ and $\widetilde{B}$ both are singular. If $\lambda$ and $\mu_1$ are zero, then the given map is {\color{red}not surjective} as the matrix of the form $\begin{pmatrix}
        c_1 & c_2\\
        c_3 & c_4
    \end{pmatrix}$ with $c_1$ or $c_2$ being non-zero are not in the image. Similarly, if $\mu$ and $\mu_2$ are zero then the matrix $\begin{pmatrix}
        0 & 0\\
        c_1 & c_2
    \end{pmatrix}$ and its conjugates are not in the image for $c_1$ and $c_2$ both being not zero at the same time. If $\mu$ and $\mu_1$ is zero then for $a \neq 0$, let $x=\begin{pmatrix}
        a_0 & a_1\\
        0 & 0
    \end{pmatrix}$ otherwise consider $x= \begin{pmatrix}
        0 & a_0\\
        0 & 1
    \end{pmatrix}$ and if $d\neq 0$, consider $y=\begin{pmatrix}
        0 & 0\\
        b_1 & b_0
    \end{pmatrix}$ otherwise $y=\begin{pmatrix}
        1 & 0\\
        b_0 & 0
    \end{pmatrix}$. In each case, depending on $a$ and $d$ being zero or non-zero, $a_0,a_1,b_0,b_1$ can be chosen for the considered $C$, such that the map is surjective. The proof is similar for $\lambda$ and $\mu_2$ being zero.
    
\subsubsection{If $\widetilde{B}=\begin{pmatrix}
        \mu_1 & 1\\
        0 & \mu_2
    \end{pmatrix}$} Here we consider both the cases $\mu_1=\mu_2$ and $\mu_1\neq \mu_2$. If $C$ is singular with $J_A$ and $\widetilde{B}$ both being non-singular. Consider $x=\begin{pmatrix}
        a_0 & 0\\
        a_1 & 0
    \end{pmatrix}$. Then $C-J_Ax^{k_1}=\begin{pmatrix}
        a-\lambda_1a_0^{k_1} & b\\
        c-\lambda_2a_1a_0^{k_1} & d
    \end{pmatrix}$. If $b$ or $d$ is non-zero then $a_0$ and $a_1$ can be chosen such that $C-J_Ax^{k_1}$ is non-singular and $\widetilde{B}$ being non-singular gives us the solution. If $b$ and $d$ both are zero then for $a\neq \mu_2c$, let $y=\begin{pmatrix}
        0 & 0\\
        0 & b_1
    \end{pmatrix}$ such that $C-\widetilde{B}y^{k_2}=\begin{pmatrix}
        a & -b_1^{k_1}\\
        c & -\mu_2 b_1^{k_1}
    \end{pmatrix}$. Choosing $b_1$ non-zero and $J_A$ being non-singular gives us the desired result as $C-\widetilde{B}y^{k_2}$ is non-singular. If $a=\mu_2 c$, let $x=\begin{pmatrix}
        a_0 & 0\\
        0 & 0
    \end{pmatrix}$ and $y=\begin{pmatrix}
        1 & 0\\
        b_1 & 0
    \end{pmatrix}$. Since $c=\mu_2b_1$ and $a=\lambda_1a_0^{k_1}+\mu_1+b_1$ have solution over $K$. Thus $C=J_Ax^{k_1}+\widetilde{B}y^{k_2}$ have solution with $C$ being singular. 
    
    We consider the case when $J_A$ is non-singular and $\widetilde{B}$ is singular i.e. one or both of 
    $\mu_1$, $\mu_2$ is zero. 
    \paragraph{For $\mu_1=0$ and $\mu_2$ being non-zero.} If $c\neq \mu_2a$, then let 
    $y=\begin{pmatrix}
            0 & 0\\
            0 & b_1
        \end{pmatrix}$. Choosing $b_1$ non-zero such that $C-\widetilde{B}y^{k_2}$ is non-singular and 
        $J_A$ being non-singular gives us the desired outcome. If $c=\mu_2 a$ then $ad-ab\mu_2=0$ as 
        $\det(C)$ is zero. Let $y=\begin{pmatrix}
            1 & 0\\
            b_1 & 0
        \end{pmatrix}$. Fix $b_1$ such that $C-\widetilde{B}y^{k_2}=\begin{pmatrix}
            a-b_1 & b\\
            a\mu_2-\mu_2b_1 & d
        \end{pmatrix}$ is non-singular and $J_A$ being non-singular gives us the required result. The case for $\mu_1 \neq 0$ and $\mu_2 =0$ is similar and we can conclude that the map is {\color{red}surjective}.
\paragraph{For $\mu_1$ and $\mu_2$ both being zero.} If $c$ or $d$ is non zero then let $y=\begin{pmatrix}
            0 & 0\\
            b_1 & b_0
        \end{pmatrix}$. Then $C-\widetilde{B}y^{k_2}$ is non-singular for some $b_0$ and $b_1$ hence $J_A^{-1}\left(C-\widetilde{B}y^{k_2}\right)=x^{k_1}$ has the solution over $K$. If $c$ and $d$ both are zero then consider $x=\begin{pmatrix}
            a_0 & 0\\
            0 & 0
        \end{pmatrix}$ and $y=\begin{pmatrix}
            0 & 0\\
            0 & b_0
        \end{pmatrix}$. $J_Ax^{k_1}+\widetilde{B}y^{k_2}$ is given by $\begin{pmatrix}
            \lambda_1a_0^{k_1} & b_0^{k_2}\\
            0 & 0
        \end{pmatrix}$. For the reason that $a=\lambda_1a_0^{k_1}$ and $b=b_0^{k_2}$ have solution over $K$ gives the solution for $C=J_Ax^{k_1}+\widetilde{B}y^{k_2}$. \\
    Moving on to the case when $J_A$ is singular and $\widetilde{B}$ is non-singular i.e. either $\lambda_1=0$ or $\lambda_2=0$. 
\paragraph{For $\lambda_1=0$.} If $a$ or $b$ is non-zero then  let $x=\begin{pmatrix}
            0& 0\\
            a_0 & a_1
        \end{pmatrix}$. Choosing appropriate $a_0$ and $a_1$ gives us $C-J_Ax^{k_1}$ to be non-singular and hence $C$ is in the required image. Consider the case when $a$ and $b$ both are zero. If $d\neq 0$, let $x=\begin{pmatrix}
            0 & 0\\
            a_1 & a_0
        \end{pmatrix}$ and $y=0$. Then choose $a_0$ and $a_1$  such that the equations $c=\lambda_2a_1a_0^{k_1-1}$ and $d=a_0^{k_1}$ have solution over $K$ and hence giving matrix solution over $\M(2,K)$. If $d=0$, let $x=\begin{pmatrix}
            1 & 0\\
            a_0 & 0
        \end{pmatrix}$ and $y=0$. Again choosing $a_0$ such that $c=\lambda_2a_0^{k_1}$ have solution over $K$ gives us $C=J_Ax^{k_1}+\widetilde{B}y^{k_2}$.
        \paragraph{For $\lambda_2=0$.} If $c$ or $d$ is non-zero then choose $x=\begin{pmatrix}
            a_0 & a_1\\
            0 & 0
        \end{pmatrix}$. $C-J_Ax^{k_1}$ is given by $\begin{pmatrix}
            a-\lambda_1a_0^{k_1} & b-\lambda_1a_0^{k_1-1}a_1\\
            c & d
        \end{pmatrix}$. Pick $a_0$ and $a_1$ such that $C-J_Ax^{k_1}$ is non-singular and since $\widetilde{B}$ is non-singular, we have $C-J_Ax^{k_1}=\widetilde{B}y^{k_2}$ for some $y\in \M(2,K)$. If $c$ and $d$ both are zero, then let $x=\begin{pmatrix}
            a_0 & 0\\
            0 & 0
        \end{pmatrix}$ and $y=\begin{pmatrix}
            1 & b_0\\
            0 & 0
        \end{pmatrix}$ gives us the result with $a=\lambda_1a_0^{k_1}+\mu_1$ and $b=\mu_1b_0$.\\
        Now move onward to the case when $J_A$ and $\widetilde{B}$ both are singular. 
    \paragraph{If $\lambda_1\neq 0$ and $\lambda_2= 0$.} We have further cases depending on whether $\mu_1$ or $\mu_2$ is zero. The equation we have in hand is given by $$C=\begin{pmatrix}
                \lambda_1 & 0\\
                0 & 0
            \end{pmatrix}x^{k_1}+\begin{pmatrix}
                \mu_1 & 1\\
                0 & \mu_2
            \end{pmatrix}y^{k_2}.$$
            \subparagraph{If $\mu_1\neq 0$ and $\mu_2=0$.}  One can observe that the matrix $C$ with $c$ or $d$ being non-zero can't be written as $J_Ax^{k_1}+\widetilde{B}y^{k_2}$ as $J_A$ and $\widetilde{B}$ can only read the first row. Let $x=\begin{pmatrix}
                a_0 & 0\\
                0 & 0
            \end{pmatrix}$ and $y=\begin{pmatrix}
                0 & b_0\\
                0 & 1
            \end{pmatrix}$. Then $J_Ax^{k_1}+\widetilde{B}y^{k_2}$ is $\begin{pmatrix}
                \lambda_1 a_0^{k_1} & \mu_1b_0+1\\
                0 & 0
            \end{pmatrix}$. Clearly $a_0$ and $b_0$ can be chosen such that $a=\lambda_1 a_0^{k_1}$ and $b=\mu_1b_0+1$ have solution for $a$ and $b$, and hence matrices of the form $\begin{pmatrix}
                a & b\\
                0 & 0
            \end{pmatrix}$ are in the image. 
            \subparagraph{If $\mu_1$ and $\mu_2$ both are zero.} Again the map is not surjective as matrices of the form $\begin{pmatrix}
                a & b\\
                c & d
            \end{pmatrix}$ with either $c$ or $d$ being non-zero are not in the image by the same argument given above. Now, let $x=\begin{pmatrix}
                a_0 & 0\\
                0 & 0
            \end{pmatrix}$ and $y=\begin{pmatrix}
                0 & 0\\
                0 & b_0
            \end{pmatrix}$. Then $J_Ax^{k_1}+\widetilde{B}y^{k_2}$ is $\begin{pmatrix}
                \lambda_1 a_0^{k_1} & b_0^{k_2}\\
                0 & 0
            \end{pmatrix}$. Picking $a_0$ and $b_0$ such that $a=\lambda_1a_0^{k_1}$ and $b=b_0^{k_2}$ have solution. Thus $\begin{pmatrix}
                a & b\\
                0 & 0
            \end{pmatrix}$ is in the image.
            \subparagraph{If $\mu_1=0$ and $\mu_2\neq 0$.} If $a$ and $d$ both are non-zero then let $x=\begin{pmatrix}
                a_0 & a_1\\
                0 & 0
            \end{pmatrix}$ and $y=\begin{pmatrix}
                0 & 0\\
                b_1 & b_0
            \end{pmatrix}$. If $a=0$ and $d\neq 0$, let $x=\begin{pmatrix}
                0 & a_0\\
                0 & 1
            \end{pmatrix}$ and $y=\begin{pmatrix}
                0 & 0\\
                b_1 & b_0
            \end{pmatrix}$ and if $a\neq 0$ and $d=0$, let $x=\begin{pmatrix}
                a_0 & a_1\\
                0 & 0
            \end{pmatrix}$ and $y=\begin{pmatrix}
                0 & a_0\\
                0 & 1
            \end{pmatrix}$. Also if $a$ and $d$ both are zero, then let $x=\begin{pmatrix}
                0 & a_0\\
                0 & 1
            \end{pmatrix}$ and $y=\begin{pmatrix}
                0 & a_0\\
                0 & 1
            \end{pmatrix}$. In each case $a_0,a_1,b_0,b_1$ can be chosen such that $C=J_Ax^{k_1}+\widetilde{B}y^{k_2}$.
    
        \paragraph{If $\lambda_1= 0$ and $\lambda_2\neq 0$.} We consider the cases when either $\mu_1$ or $\mu_2$ or both are zero. The equation is $$C=\begin{pmatrix}
            0 & 0\\
            0 & \lambda_2
        \end{pmatrix}x^{k_!}+\begin{pmatrix}
            \mu_1 & 1\\
            0 & \mu_2
        \end{pmatrix}y^{k_2}.$$
        
\subparagraph{If $\mu_1\neq 0$ and $\mu_2=0$.} If $d$ is non-zero then choose $x=\begin{pmatrix}
                0 & 0\\
                a_1 & a_0
            \end{pmatrix}$. Then $a_0$ and $a_1$ can be chosen such that $C-J_Ax^{k_1}=\begin{pmatrix}
            a & b\\
            0 & 0
            \end{pmatrix}$. Additionally, if $a\neq 0$, let $y=\begin{pmatrix}
                b_0 & b_1\\
                0 & 0
            \end{pmatrix}$. Then $b_0$ and $b_1$ can be chosen such that $C-J_Ax^{k_1}=\widetilde{B}y^{k_2}$. For $a=0$, let $y=\begin{pmatrix}
                0 & b_0\\
                0 & 1
            \end{pmatrix}$ such that $\begin{pmatrix}
                0 & b\\
                0 & 0
            \end{pmatrix}=\begin{pmatrix}
                0 & \mu_1b_0+1\\
                0 & 0
            \end{pmatrix}$. Then $b_0=\mu_1^{-1}\left(b-1\right)$. If $d=0$, then let $x=\begin{pmatrix}
                1 & 0\\
                a_0 & 0
            \end{pmatrix}$. Then $a_0$ can be chosen such that $C-J_Ax^{k_1}$ is $\begin{pmatrix}
                a & b\\
                0 & 0
            \end{pmatrix}$ which can be written as $\widetilde{B}y^{k_2}$ for some $y$ as we did before. Hence the map is {\color{red} surjective}. The case for $\mu_1 =0$ and $\mu_2\neq 0$
            follows a similar argument and the map is {\color{red} surjective} as well.
{\color{black}\subparagraph{If $\mu_1$ and $\mu_2$ both are zero.} The given map is surjective. The proof is similar to the above part with $\mu_2$ being zero.}
        Here, we conclude that the given map is surjective if and only if the same row of $J_A$ and $\widetilde{B}$ considered, are not entirely zero.

  \subsubsection{In this part, we consider $\widetilde{B}$ given by $\begin{pmatrix}
        \mu_1 & 0\\
        1 & \mu_2
    \end{pmatrix}$ with $\mu_1=\mu_2$ or $\mu_1\neq \mu_2$ }. When $J_A$ and $\widetilde{B}$ both are non-singular and $C$ is singular. If $b$ or $d$ is non-zero then let $x=\begin{pmatrix}
        a_0 & 0\\
        a_1 & 0
    \end{pmatrix}$. Choose $a_0$ and $a_1$ such that $C-J_Ax^{k_1}$ is non-singular and $\widetilde{B}$ being non-singular gives us the solution. If $b$ and $d$ both are zero then let $x=\begin{pmatrix}
        a_0 & 0\\
        0 & 0
    \end{pmatrix}$ and $y=\begin{pmatrix}
        1 & 0\\
        b_0 & 0
    \end{pmatrix}$. Choose $a_0$ and $b_0$ satisfying $a=\lambda a_0^{k_1}+\mu_1$ and $c=1+\mu_2b_0$ so that $C=J_Ax^{k_1}+\widetilde{B}y^{k_2}$. Consider when $J_A$ is non-singular and $\widetilde{B}$ is singular. Here, we have either $\mu_1=0$ or $\mu_2=0$ or both of them are zero.
    \paragraph{If $\mu_1=0$ and $\mu_2$ is non-zero}. If $a$ or $b$ is non-zero, let $y=\begin{pmatrix}
        0 & 0\\
        b_0 & b_1
    \end{pmatrix}$. Fix $b_0$ and $b_1$ such that $C-\widetilde{B}y^{k_2}$ is non-singular and $J_A$ being non-singular gives us the desired result. Now, if $a$ and $b$ both are zero, then let $x=\begin{pmatrix}
        0 & 0\\
        0 & a_0
    \end{pmatrix}$ and $y=\begin{pmatrix}
        1 & 0\\
        b_0 & 0
    \end{pmatrix}$. Choose $a_0$ and $b_0$ satisfying the equations $c=1+\mu_2b_0$ and $d=\mu a_0^{k_1}$ and hence we have $C=J_Ax^{k_1}+\widetilde{B}y^{k_2}$.
    \paragraph{If $\mu_1$ is non-zero and $\mu_2=0$} For $b\neq \mu_1d$, let $y=\begin{pmatrix}
        1 & 0\\
        0 & 0
    \end{pmatrix}$. Then $C-\widetilde B y^{k_2}$ is $\begin{pmatrix}
        a-\mu_1 & b\\
        c-1 & d
    \end{pmatrix}$ which is non-singular and $J_A$ is also non-singular and hence giving the aimed solution. If $b=\mu_1d$ with $a\neq 0$, then $x=\begin{pmatrix}
        a_0 & 0\\
        a_1 & 0
    \end{pmatrix}$ with $y=\begin{pmatrix}
        0 & d\\
        0 & 1
    \end{pmatrix}$ gives us the desired result where $a_0$ and $a_1$ satisfies the equations $a=\lambda a_0^{k_1}$ and $c=\mu a_1a_0^{k_1-1}$. For $b=\mu_1d$ and $a\neq 0$, let $y=\begin{pmatrix}
        0 & d\\
        0 & 1
    \end{pmatrix}$. Then $C-\widetilde By^{k_2}$ is $\begin{pmatrix}
        a & 0\\
        c & 0
    \end{pmatrix}$ which is in the image of $J_Ax^{k_1}$ by taking $x=\begin{pmatrix}
        a_0 & 0\\
        a_1 & 0
    \end{pmatrix}$ such that $a_0$ and $a_1$ satisfies $a=\lambda a_0^{k_1}$ and $c=\mu a_1 a_0^{k_1-1}$. If $a=0$ with $b=\mu_1 d$, let $x=\begin{pmatrix}
        0 & 0\\
        a_1 & a_0
    \end{pmatrix}$ such that $C-J_Ax^{k_1}$ is $\begin{pmatrix}
        0 & \mu_1d\\
        c-\mu a_1a_0^{k_1-1} & d-\mu a_0^{k_1}
    \end{pmatrix}$. Choose $a_0 \neq 0$, such that $d-\mu a_0^{k_1}\neq 0$ and hence $a_1$ can be chosen such that $c=\mu a_1a_0^{k_1-1}$. Then $C-J_Ax^{k_1}$ becomes $\begin{pmatrix}
        0 & \mu_1d\\
        0 & d-\mu a_0^{k_1}
    \end{pmatrix}$. Let $y=\begin{pmatrix}
        0 & b_1\\
        0 & b_0
    \end{pmatrix}$. Again we can choose $b_0$ and $b_1$ such that $C-J_Ax^{k_1}$ is $\widetilde By^{k_2}$.
    \paragraph{If $\mu_1$ and $\mu_2$ both are zero}If $a$ or $b$ is non-zero, then let $y=\begin{pmatrix}
        b_0 & b_1\\
        0 & 0
    \end{pmatrix}$. Considering the $y$ mentioned, we can choose $b_0$ and $b_1$ such that $C-\widetilde B y^{k_2}$ is non-singular and hence lies in the image $J_Ax^{k_1}$. If $a$ and $b$ both are zero, then letting $x=\begin{pmatrix}
        0 & 0\\
        0 & a_0
    \end{pmatrix}$ and $y=\begin{pmatrix}
        b_0 & 0\\
        0 & 0
    \end{pmatrix}$ gives us $J_Ax^{k_1}+\widetilde B y^{k_2}=\begin{pmatrix}
        0 & 0\\
        b_0^{k_2} & \mu a_0^{k_1}
    \end{pmatrix}$ and choosing $a_0$ and $b_0$ such that $c=b_0^{k_2}$ and $d=\mu a_0^{k_1}$ gives us the result.\\
    We consider the case when $J_A$ is singular and $\widetilde B$ is non-singular. Here, we have $\lambda_1$ is zero or $\lambda_2$ is zero.
    \paragraph{For $\lambda_1=0$}This case is covered in $3.2.3.3$ as we only use the fact that $\widetilde B$ is non-singular.
    \paragraph{For $\lambda_2=0$}As did in $3.2.3.4$, if $c$ or $d$ is non-zero, we are done. Assume $c$ and $d$ both are zero. If $a\neq 0$, let $x=\begin{pmatrix}
        a_0 & a_1\\
        0 & 0
    \end{pmatrix}$ with $y=0$ and if $a=0$, then let $x=\begin{pmatrix}
        0 & a_0\\
        0 & 1
    \end{pmatrix}$ with $y=0$. In each case, choose $a_0$ such that $C=J_Ax^{k_1}+\widetilde B y^{k_2}$.
    Now we examine the case when $J_A$ and $\widetilde{B}$ both are singular.
    \paragraph{For $\lambda_1\neq 0$ and $\lambda_2=0$}The equation in hand is $$\begin{pmatrix}
        a & b\\
        c & d
    \end{pmatrix}=\begin{pmatrix}
        \lambda_1 & 0\\
      0 & 0
    \end{pmatrix}x^{k_1}+\begin{pmatrix}
        \mu_1 & 0\\
        1 & \mu_2
    \end{pmatrix}y^{k_2}.$$ Since $\widetilde B$ is also singular, so we further have $\mu_1$ or $\mu_2$ being zero.
    \subparagraph{For $\mu_1\neq 0$ and $\mu_2=0$}If $c\neq 0$, then let $y=\begin{pmatrix}
        b_0 & b_1\\
        0 & 0
    \end{pmatrix}$ such that $C-\widetilde By^{k_2}$ is $\begin{pmatrix}
        a-\mu_1b_0^{k_2} & b-\mu_1b_1b_0^{k_2-1}\\
        c-b_0^{k_2} & c-b_1b_0^{k_2-1}
    \end{pmatrix}$. Choose $b_0$ such that $c-b_0^{k_2}=0$ and hence choosing $b_1$ such that $c-b_1b_0^{k_2-1}=0$. Now if $a-\mu_1b_0^{k_2}\neq 0$, let $x=\begin{pmatrix}
        a_0 & a_1\\
        0 & 0
    \end{pmatrix}$ otherwise let $x=\begin{pmatrix}
        0 & a_0\\
        0 & 1
    \end{pmatrix}$ such that $\begin{pmatrix}
       a-\mu_1b_0^{k_2} &  b-\mu_1b_1b_0^{k_2-1}\\
       0 & 0
    \end{pmatrix}$ can be written as $J_Ax^{k_1}$ by choosing appropriate $a_0$ and $a_1$ in each case. If $c=0$, let $y=\begin{pmatrix}
        0 & d\\
        0 & 1
    \end{pmatrix}$ such that $C-\widetilde B y^{k_2}$ is $\begin{pmatrix}
        a & b-\mu_1 d\\
        0 & 0
    \end{pmatrix}$. The matrix we have in hand is again in the image of $J_Ax^{k_1}$ by letting $x$ as we did before for $c\neq 0$.
    \subparagraph{For $\mu_1=0$ and $\mu_2\neq 0$} If $c\neq 0$, let $y=\begin{pmatrix}
        b_0 & b_1\\
        0 & 0
    \end{pmatrix}$ otherwise let $y=\begin{pmatrix}
        0 & b_0\\
        0  & 1
    \end{pmatrix}$. In each case $b_0$ and $b_1$ can be chosen such that $C-\widetilde By^{k_2}$ is given by $\begin{pmatrix}
        a & b\\
        0 & 0
    \end{pmatrix}$ which lies in $J_Ax^{k_1}$ as done for $\mu_2=0$.
We note that the same proof works when both $\mu_1$ and $\mu_2$ both are zero.

\paragraph{For $\lambda_1=0$ and $\lambda_2\neq 0$}. the equation is given by $$\begin{pmatrix}
    a & b\\
    c & d
\end{pmatrix}=\begin{pmatrix}
    0 & 0\\
    0 & \lambda_2
\end{pmatrix}x^{k_1}+\begin{pmatrix}
    \mu_1 & 0\\
    1 & \mu_2
\end{pmatrix}y^{k_2}.$$
\subparagraph{If $\mu_1\neq 0$ and $\mu_2=0$}For $a\neq 0$, let $y=\begin{pmatrix}
    b_0 & b_1\\
    0 & 0
\end{pmatrix}$ and for $a=0$, let $y=\begin{pmatrix}
    0 & b_0\\
    0 & 1
\end{pmatrix}$. We get $C-\widetilde By^{k_2}$ as $\begin{pmatrix}
    a-\mu_1b_0^{k_2} & b-\mu_1b_1b_0^{k_2-1}\\
    c-b_0^{k_2} & d-b_1b_0^{k_2-1}
\end{pmatrix}$ or $\begin{pmatrix}
    0 & b-\mu_1b_0\\
    c & d-b_0
\end{pmatrix}$ in respective cases. For $a\neq 0$, choose $b_0$ such that $a-\mu_1b_0^{k_2}=0$ and hence choosing $b_1$ such that $b-\mu_1b_1b_0^{k_2-1}=0$. For $a=0$, choose $b_0$ such that $b-\mu_1b_0=0$. Thus, we get $\begin{pmatrix}
     0 & 0\\
     c-b_0^{k_2} & d-b_1b_0^{k_2-1}\end{pmatrix}$ or $\begin{pmatrix}
         0 & 0\\
         c & d-b_0
     \end{pmatrix}$. If $d-b_1b_0^{k_2-1}\neq$ for $a\neq 0$ and $d-b_0\neq 0$ for $a=0$, we let $x=\begin{pmatrix}
         0 & 0\\
         a_1 & a_0
     \end{pmatrix}$. Then choosing $a_0$ and $a_1$ in each case gives us the solution for matrices obtained to be in the image of $J_Ax^{k_1}$. If $d-b_1b_0^{k_2-1}=0$ for $a\neq 0$ and $d-b_0= 0$ for $a=0$, considering $x=\begin{pmatrix}
         1 & 0\\
         a_0 & 0
     \end{pmatrix}$ gives us the solution by choosing appropriate $a_0$ in each case. 
     \subparagraph{If $\mu_1=0$}In this case, $\mu_2\in k$. Since the first row of $J_A$ and $\widetilde B$ is entirely zero, so the given map is not surjective as the matrices of the form $\begin{pmatrix}
       a & b\\
       c & d
     \end{pmatrix}$ with $a$ or $b$ being non-zero are not in the image.
Hence the map is \textcolor{red}{not surjective} when $\lambda_1$ and $\mu_1$ both are zero. 
\subsubsection{Consider $\widetilde B$ given by $\begin{pmatrix}
    z_1  & z_2\\
    1 & z_3
\end{pmatrix}_{z_i\neq 0}$}
If $C$ is singular with $J_A$ and $\widetilde B$ both being non-singular, let $x=\begin{pmatrix}
    a_0 & a_1\\
    0  & 0
\end{pmatrix}$. Then $C-J_Ax^{k_1}=\begin{pmatrix}
    a-\lambda_1a_0^{k_1} & b-\lambda_1a_1a_0^{k_1-1}\\
    c & d
\end{pmatrix}$. The $\det \left(C-J_Ax^{k_1}\right)$ is $-\lambda_1a_0^{k_1}d+\lambda_1a_1a_0^{k_1-1}c.$ If $c$ or $d$ is non-zero, choose $a_0$ and $a_1$ such that $\det \left(C-J_Ax^{k_1}\right)$ is non-zero and hence $\widetilde B$ being non-singular gives us the solution. If $c$ and $d$ both are zero, we are left to show that the matrix of the form $U:=\begin{pmatrix}
    a & b\\
    0 & 0
\end{pmatrix}$ is in the image of $J_Ax^{k_1}+\widetilde By^{k_2}.$ Let $x=\begin{pmatrix}
    a_0 & a_1\\
    0 & a_0
\end{pmatrix}$. Then $U-J_Ax^{k_1}$ is given by $\begin{pmatrix}
    a-\lambda_1a_0^{k_1} & b-k_1a_0^{k_1-1}a_1\\
    0 & -\lambda_2a_0^{k_1}
\end{pmatrix}$. Choose $a_0\neq 0$ such that $a-\lambda_1a_0^{k_1}\neq 0$ and hence making $U-J_Ax^{k_1}$ non-singular. With $\widetilde B$ being non-singular, we have $C$ to be in the image of the polynomial. Now, move onto the case when $J_A$ is singular and $\widetilde B$ is non-singular i.e. for $J_A$, we have either $\lambda_1=0$ or $\lambda_2=0$.
\paragraph{For $\lambda_1\neq0$ and $\lambda_2=0$}Let $x=\begin{pmatrix}
    a_0 & a_1\\
    0  & 0
\end{pmatrix}$. As done above, we have the solution with chosen $x$ for $c$ or $d$ being non-zero. If $c$ and $d$ both are zero, then for $a\neq 0$, let $x=\begin{pmatrix}
    a_0 & a_1\\
    0 & 0
\end{pmatrix}$ and $y=0$. Choose $a_0$ satisfying $a=\lambda_1a_0^{k_1}$ and hence choosing $a_1$ such that $b=\lambda_1a_1a_0^{k_1-1}$. If $a=0$, then let $x=\begin{pmatrix}
    0 & a_0\\
    0 & 1
\end{pmatrix}$ and $y=0$ gives us the result by choosing $a_0$ satisfying $b=\lambda_1a_0$.
\paragraph{For $\lambda_1=0$ and $\lambda_2\neq 0$}. Let $x=\begin{pmatrix}
    0 & 0\\
    a_1 & a_0
\end{pmatrix}$ such that $C-J_Ax^{k_1}$ is given by $\begin{pmatrix}
    a & b\\
    c-\lambda_2a_1a_0^{k_1-1} & d-\lambda_2a_0^{k_1}
\end{pmatrix}$. If $a$ or $b$ is non-zero, choose $a_0$ and $a_1$ such that $C-J_Ax^{k_1}$ is non-singular and hence $\widetilde B$ being non-singular gives us the solution for some $y\in M(2,k).$ Now, suppose $a$ and $b$ both are zero with $d\neq 0$, then $a_0$ is chosen such that $d=\lambda_2a_0^{k_1}$ and $a_1$ is chosen such that $c=\lambda_2a_1a_0^{k_1-1}$. If $a,b$ and $d$ are zero, letting $x=\begin{pmatrix}
    1 & 0\\
    a_0 & 0
\end{pmatrix}$ with $y=0$ gives us the map to be {\color{red}surjective} where $a_0$ satisfies $c=\lambda_2a_0$.

Consider the case when $J_A$ is non-singular and $\widetilde B$ is singular i.e. $z_2=z_1z_3$.  Equation is given by $$\begin{pmatrix}
    a & b\\
    c& d
\end{pmatrix}=\begin{pmatrix}
    \lambda_1 & 0\\
    0 & \lambda_2
\end{pmatrix}x^{k_1}+\begin{pmatrix}
    z_1 & z_1z_3\\
    1 & z_3
\end{pmatrix}y^{k_2}.$$ 
If $c\neq 0$, let $y=\begin{pmatrix}
    b_0 & b_1\\
    0 & 0
\end{pmatrix}$. Then $C-\widetilde By^{k_2}$ is given by $\begin{pmatrix}
    a-z_1b_0^{k_2} & b-z_1b_1b_0^{k_1-1}\\
    c-b_0^{k_2}& d-b_1b_0^{k_2-1}
\end{pmatrix}$. Choose $b_0$ such that $c-b_0^{k_2}=0$ and hence choose $b_1$ such that $b-z_1b_1b_0^{k_1-1}=0$. Then we have a diagonal matrix which is in the image of $J_Ax^{k_1}$ where $x$ is $\begin{pmatrix}
    a_0 & 0\\
    0 & a_1
\end{pmatrix}$. $a_0$ satisfies $a-z_1b_0^{k_2}=\lambda_1a_0^{k_1}$ and $a_1$ satisfies $d-b_1b_0^{k_2-1}=\lambda_2a_1^{k_1}$.  If $c$ is zero, then $ad=0$ as $\det$ of $C$ is zero. Suppose $a=0$ and $d\neq 0$, let $ y=0 $ and $x=\begin{pmatrix}
    0 & a_1\\
    0  & a_0
\end{pmatrix}$ where $a_0$ satisfies $d=\lambda_2a_0^{k_1}$ and $a_1$ satisfies $b=\lambda_1a_1a_0^{k_1-1}$. Consider $a\neq 0$ and $d=0$, then letting $y=0$ and $x=\begin{pmatrix}
    a_0 & a_1\\
    0 & 0
\end{pmatrix}$ gives us the solution where $a=\lambda_1a_0^{k_1}$  and $b=\lambda_1a_1a_0^{k_1-1}$. If $a$ and $d$ both are zero, then let $x=\begin{pmatrix}
    0 & a_0\\
    0 & 1
\end{pmatrix}$ and let $y=\begin{pmatrix}
    0 & b_0\\
    0 & 1
\end{pmatrix}$. Then $J_Ax^{k_1}+\widetilde By^{k_2}$ is $\begin{pmatrix}
    0 & \lambda_1a_0+z_1b_0+z_1z_3\\
    0 & \lambda_2+b_0+z_3
\end{pmatrix}$. Choose $b_0=-\lambda_2-z_3$ and hence choose $a_0$ such that $b=\lambda_1a_0+z_1b_0+z_1z_3$ and hence we have the {\color{red}surjectivity} of the given map.

If $J_A$ and $\widetilde B$ both are singular then we have two possibilities, as discussed below:
\paragraph{The first case considers $\lambda_1\neq 0$ and $\lambda_2=0$ with $z_2=z_1z_3$} The equation is given by $$\begin{pmatrix}
    a & b\\
    c & d
\end{pmatrix}=\begin{pmatrix}
    \lambda_1 & 0\\
    0 & 0
\end{pmatrix}x^{k_1}+\begin{pmatrix}
    z_1 & z_1z_3\\
    1 & z_3
\end{pmatrix}y^{k_2}.$$ For $c\neq 0$, let $y=\begin{pmatrix}
    b_0 & b_1\\
    0 & 0
\end{pmatrix}$ such that $C-\widetilde By^{k_2}$ is $\begin{pmatrix}
    a-z_1b_0^{k_2} & b-z_1b_1b_0^{k_2-1}\\
    c-b_0^{k_2} & d-b_1b_0^{k_2}
\end{pmatrix}$. Choose $b_0$ such that $c-b_0^{k_2}=0$ and hence choosing $b_1$ such that $d-b_1b_0^{k_2}=0$. For $a-z_1b_0^{k_2}\neq 0$, let $x=\begin{pmatrix}
    a_0 & a_1\\
    0 & 0
\end{pmatrix}$  such that $a_0$ satisfies $a-z_1b_0^{k_2}=\lambda_1a_0^{k_1}$ and $a_1$ satisfy $b-z_1b_1b_0^{k_2-1}=\lambda_1a_1a_0^{k_1-1}$. For $a-z_1b_0^{k_2}=0$,  let $x=\begin{pmatrix}
    0 & a_0\\
    0 & 1
\end{pmatrix}$ such that $a_0$ is chosen such that it satisfies the equation $b-z_1b_1b_0^{k_2-1}=\lambda_1a_0$ and hence for $c\neq 0$, we have solution. We consider the case when $c$ is zero. Let $y=\begin{pmatrix}
    0 & 0\\
    0 & b_0
\end{pmatrix}$ such that $C-\widetilde By^{k_2}$ is given by $\begin{pmatrix}
    a & b-z_1z_3b_0^{k_2}\\
    0 & d-z_3b_0^{k_2}
\end{pmatrix}$. Again choose $b_0$ such that $d-z_3b_0^{k_2}=0$ and hence letting $x$ as we did for $c\neq 0$, gives us the map to be {\ color {red} surjective}.
\paragraph{ Second case consider $\lambda_1=0$ and $\lambda_2\neq$ with $z_2=z_1z_3$} The equation we are dealing here is $$\begin{pmatrix}
    a & b\\
    c & d
\end{pmatrix}=\begin{pmatrix}
    0 & 0\\
    0 & \lambda_2
\end{pmatrix}x^{k_1}+\begin{pmatrix}
    z_1 & z_1z_3\\
    1 & z_3
\end{pmatrix}.$$ We let $y=\begin{pmatrix}
    b_0 & b_1\\
    0 & 0
\end{pmatrix}$ for $a\neq 0$. Then $C-\widetilde By^{k_2}$ is $\begin{pmatrix}
    a-z_1b_0^{k_2} & b-z_1b_1b_0^{k_2-1}\\
    c-b_0^{k_2} & d-b_1b_0^{k_2}
\end{pmatrix}$. Choose $b_0$ such that $a-z_1b_0^{k_2}=0$ and hence choosing $b_1$ such that $b-z_1b_1b_0^{k_2-1}=0$. For $d-b_1b_0^{k_2}\neq 0$, let $x=\begin{pmatrix}
    0 & 0\\
    a_1 & a_0
\end{pmatrix}$  such that $a_0$ satisfies $d-b_1b_0^{k_2}=\lambda_2a_0^{k_1}$ and $a_1$ satisfy $c-b_0^{k_2}=\lambda_2a_1a_0^{k_1-1}$. For $d-b_1b_0^{k_2}=0$,  let $x=\begin{pmatrix}
    1 & 0\\
    a_0 & 0
\end{pmatrix}$ such that $a_0$ satisfies the equation $ c-b_0^{k_2}=\lambda_2a_0$ and hence for $c\neq 0$, we have solution. If $a$ is zero. Let $y=\begin{pmatrix}
    0 & b_0\\
    0 & 1
\end{pmatrix}$ such that $C-\widetilde By^{k_2}$ is given by $\begin{pmatrix}
    0 & b-z_1b_0-z_1z_3\\
    c & d-b_0-z_3
\end{pmatrix}$. Again choose $b_0$ such that $b-z_1b_0-z_1z_3=0$ and hence letting $x$ as we did for $a\neq 0$, gives us the solution.
{The map is \textcolor{red}{surjective} for $\widetilde B =\begin{pmatrix}
    z_1 & z_2\\
    1 & z_3
\end{pmatrix}_{z_i\neq 0}$.}
\subsubsection{We have $\widetilde B$ given by $\begin{pmatrix}
    0 & z_1\\
    1 & z_2
\end{pmatrix}_{z_2\neq 0}$}We start with the case when $C$ is singular with $J_A$ and $\widetilde B$ both being non-singular. The proof is similar to $C$ being singular with $J_A$ and $\widetilde B$ both being non-singular. Examine the case when $J_A$ is non-singular and $\widetilde B$ being singular. The equation is given by $$\begin{pmatrix}
    a & b\\
    c & d
\end{pmatrix}=\begin{pmatrix}
    \lambda_1 & 0\\
    0 & \lambda_2
\end{pmatrix}x^{k_1}+\begin{pmatrix}
    0 & 0\\
    1 & z_2
\end{pmatrix}y^{k_2}.$$ If $a\neq 0$, let $x=\begin{pmatrix}
    a_0 & a_1\\
    0 & 0
\end{pmatrix}$. Then $C-J_Ax^{k_1}$ is given by $\begin{pmatrix}
    a-\lambda_1a_0^{k_1} & b-\lambda_1a_1a_0^{k_1-1}\\
    c & d
 \end{pmatrix}$. Choose $a_0$ such that $a-\lambda_1a_0^{k_1}=0$ and hence choose $a_1$ such that $b-\lambda_1a_1a_0^{k_1-1}=0$. The aim is to show a matrix of the form $\begin{pmatrix}
    0 & 0\\
    c  & d
\end{pmatrix}$ are in the image of $\widetilde By^{k_2}$. For $c\neq 0$, let $y=\begin{pmatrix}
  b_0 & b_1\\
  0 & 0
\end{pmatrix}$. Then $\widetilde By^{k_2}$ is $\begin{pmatrix}
    0 & 0\\
    b_0^{k_2} & b_1b_0^{k_2-1}
\end{pmatrix}$. Then choose $b_0$ such that $c=b_0^{k_2}$ and $b_1$ such that $d=b_1b_0^{k_2-1}$ and hence $C$ is in the image. If $c=0$, then let $y=\begin{pmatrix}
    0 & b_0\\
    0 & 1
\end{pmatrix}$, where $b_0$ satisfies $d=b_0+z_2$. If $a$ is zero, let $x=\begin{pmatrix}
    0 & a_0\\
    0 & 1
\end{pmatrix}$. Then $C-J_Ax^{k_1}$ is $\begin{pmatrix}
    0 & b-\lambda_1a_0\\
    c & d-\lambda_2
\end{pmatrix}$. Let $a_0$ be such that $b-\lambda_1a_0=0$ and hence again we land on to show that matrix of the form $\begin{pmatrix}
    0 & 0\\
    c & d-\lambda_2
\end{pmatrix}$ is in the image of $\widetilde B y^{k_2}$, which we already did above. Thus the map is {\color{red}surjective}.

Now we deal with the case when $J_A$ is singular and $\widetilde B$ is non-singular i.e. either $\lambda_1=0$ or $\lambda_2=0$.
\paragraph{For $\lambda_1\neq 0$ and $\lambda_2=0$} If $c$ or $d$ is non-zero, then let $x=\begin{pmatrix}
    a_0 & a_1\\
    0 & 0
\end{pmatrix}$ such that by choosing $a_0$ and $a_1$, the matrix $C-J_Ax^{k_1}$ becomes non-singular and $\widetilde B$ being non-singular gives us the result. If $c$ and $d$ both are zero, then for $a\neq 0$, let $x=\begin{pmatrix}
    a_0 & a_1\\
    0 & 0
\end{pmatrix}$ such that $a=\lambda_1a_0^{k_1}$ and $b=\lambda_1a_1a_0^{k_1-1}$ and for $a=0$, let $x=\begin{pmatrix}
    0 & a_0\\
    0 & 1
\end{pmatrix}$ such that $b=\lambda_1a_0$ and hence $C$ is in the image. Hence the map is {\color{red}surjective}.
\paragraph{For $\lambda_1=0$ and $\lambda_2\neq 0$}If $a$ or $b$ is non-zero, let $x=\begin{pmatrix}
    0 & 0\\
    a_1 & a_0
\end{pmatrix}$. The $\det\left(C-J_Ax^{k_1}\right)$ is $-a\lambda_2a_0^{k_1}+b\lambda_2a_1a_0^{k_1-1}$ which can be made non-zero by choosing appropriate $a_0$ and $a_1$. Considering $C-J_Ax^{k_1}$ being non-singular and $\widetilde B$ also being non-singular gives us the solution. If $a$ and $b$ both are zero, then $\begin{pmatrix}
    0 & 0\\
    c & d
\end{pmatrix}$ is in the image by choosing $x=\begin{pmatrix}
    0 & 0\\
    a_1 & a_0
\end{pmatrix}$ for $d\neq 0$ and $\begin{pmatrix}
    1 & 0\\
    a_0 & 0
\end{pmatrix}$ for $d=0$. Choose $a_0$ and $a_1$ according to the equations and we get the {\color{red}surjectivity} of the map.

Consider the case when both $J_A$ and $\widetilde B$ are singular matrices. The singularity of both cases gives us $z_1=0$ with either $\lambda_1=0$ or $\lambda_2=0$.
\paragraph{For $\lambda_1=0$ and $\lambda_2\neq0$} The equation we have is $$\begin{pmatrix}
    a & b\\
    c & d
\end{pmatrix}=\begin{pmatrix}
    0 & 0\\
    0 & \lambda_2
\end{pmatrix}x^{k_1}+\begin{pmatrix}
    0 & 0\\
    1 & z_2
\end{pmatrix}y^{k_2}.$$ The given map is \textcolor{red}{not surjective} as the matrix of the form $\begin{pmatrix}
    a & b \\
    c &d
\end{pmatrix}$ with $a$ or $b$ being non-zero, are not in the image. 
\paragraph{For $\lambda_1\neq 0$ and $\lambda_2=0$} The equation is $$\begin{pmatrix}
    a & b\\
    c & d
\end{pmatrix}=\begin{pmatrix}
    \lambda_1 & 0\\
    0 & 0
\end{pmatrix}x^{k_1}+\begin{pmatrix}
    0 & 0\\
    1 & z_2
\end{pmatrix}y^{k_2}.$$ If $a\neq 0$, let $x=\begin{pmatrix}
    a_0 & a_1\\
    0 & 0
\end{pmatrix}$ otherwise, let $x=\begin{pmatrix}
    0 & a_0\\
    0 & 1
\end{pmatrix}$. For $a\neq 0$, choose $a_0$ such that $a=\lambda_1a_0^{k_1}$ and hence $a_1$ such that $b=\lambda_1a_1a_0^{k_1-1}$. For $a=0$, choose $a_0$ such that $b=\lambda_1a_0$. Now, it suffices to show $\begin{pmatrix}
    0 & 0\\
    c & d
\end{pmatrix}$ is in the image of $\widetilde By^{k_2}$, which we already did when only $\widetilde B$ is considered to be singular.
{Hence, the map is \textcolor{red}{surjective} if and only if $\lambda_1$ is not zero. }
\subsubsection{The representative that we now have is given by $\begin{pmatrix}
    z_1 & z_2\\
    1 & 0
\end{pmatrix}_{z_1\neq 0}$}The map is surjective for $C$ being singular with $J_A$ and $\widetilde B$ both being non-singular. The proof follows from the above case. Consider the case when $J_A$ is singular and $\widetilde B$ is singular i.e. $z_2=0$. If $b$ or $d$ is non-zero, then let $y=\begin{pmatrix}
    b_0 & 0\\
    b_1 & 0
\end{pmatrix}$. $b_0$ and $b_1$ can be chosen such that $C-\widetilde By^{k_2}$ is non-singular and $J_A$ being non-singular gives us the solution. If $b$ and $d$ both are zero, then let $x=\begin{pmatrix}
    a_0 & 0\\
    0 & 0
\end{pmatrix}$ and $y\begin{pmatrix}
    b_0 & 0 \\
    0 & 0
\end{pmatrix}$. Then $J_Ax^{k_1}+\widetilde By^{k_2}$ is given by $\begin{pmatrix}
    \lambda_1a_0^{k_1} +z_1b_0^{k_2} & 0\\
    b_0^{k_2} & 0
\end{pmatrix}$. Choose $b_0$ such that $c=b_0^{k_2}$ and hence choose $a_0$ such that $a=\lambda_1a_0^{k_1} +z_1b_0^{k_2}$. Thus, the map is {\color{red}surjective}. 

Consider the case when $J_A$ is singular and $\widetilde B$ is non-singular. For $J_A$ being singular, we have either $\lambda_1=0$ or $\lambda_2=0$.
\paragraph{For $\lambda_1=0$ and $\lambda_2\neq 0$} If $a$ or $b$ is non-zero, consider $x=\begin{pmatrix}
    0 & 0\\
    a_1 & a_0
\end{pmatrix}$. With the same argument as before, we choose $a_0$ and $a_1$ such that $\det\left(C-J_Ax^{k_1}\right)$ is non-zero and $\widetilde B$ being non-singular gives us the required $y$. If $a$ and $b$ both are zero, with the same $x$ chosen and $y=0$ gives us the solution for $C$ if $d\neq 0$ by choosing $a_0$ such that $d=\lambda_2a_0^{k_1}$ and $a_1$ such that $b=\lambda_2a_1a_0^{k_1-1}$. For $d=0$, let $x=\begin{pmatrix}
    1 &0\\
    a_0 & 0
\end{pmatrix}$ where $a_0$ satisfies $c=\lambda_2a_0$ and let $y=0$. Therefore, the given map is {\color{red}surjective}.
\paragraph{For $\lambda_1\neq 0$ and $\lambda_2=0$}For $c$ or $d$ being non-zero, let $x=\begin{pmatrix}
    a_0 & a_1\\
    0 & 0
\end{pmatrix}$. Choose $a_0$ and $a_1$ such that $C-J_Ax^{k_1}$ is non-singular and $\widetilde B$ being non-singular gives us the solution. If $c$ and $d$ both are zero, for $a\neq 0$, let $x=\begin{pmatrix}
    a_0 & a_1\\
    0 & 0
\end{pmatrix}$ with $y=0$, where $a_0$ satisfies $a=\lambda_1a_0^{k_1}$ and $a_1$ satisfy $b=\lambda_1a_1a_0^{k_1-1}$. For $a=0$, let $x=\begin{pmatrix}
    0 & a_0\\
    0 & 1
\end{pmatrix}$ where $b=\lambda_1a_0$ and $y=0$. Hence, the given map is {\color{red}surjective}.

Consider the case when both $J_A$ and $\widetilde B$ are singular. We have two cases :
\paragraph{$\lambda_1=0$ with $z_1=0$}the given map is \textcolor{red}{not surjective} as matrices having non-zero entries in the first row are not in the image.
\paragraph{$\lambda_2=0$ with $z_1=0$}If $a\neq 0$, let $x=\begin{pmatrix}
    a_0 & a_1\\
    0 & 0
\end{pmatrix}$ where $a_0$ satisfy $a=\lambda_1a_0^{k_1}$ and $a_1$ satisfies $b=\lambda_!a_1a_0^{k_1-1}$. Otherwise, let $x=\begin{pmatrix}
    0 & a_0\\
    0 & 1
\end{pmatrix}$ where $a_0$ satisfy $b=\lambda_1a_0$. Then $C-J_Ax^{k_1}$ is $\begin{pmatrix}
    0 & 0\\
    c & d
\end{pmatrix}$. If $c\neq 0$, let $y=\begin{pmatrix}
    b_0 & b_1\\
    0 & 0
\end{pmatrix}$ such that $c=b_0^{k_2}$ and $d=b_1b_0^{k_2-1}$. Otherwise, let $y=\begin{pmatrix}
    0 & 0\\
    0 & b_0
\end{pmatrix}$. Then choose $b_0$ such that $d=z_2b_0^{k_2}$.\\
The given map is \textcolor{red}{surjective if and only if either of $\lambda_1$ or $z_1$ is non-zero}.
\subsubsection{Consider $\widetilde B$ given by $\begin{pmatrix}
    0 & z_2\\
    1 & 0
\end{pmatrix}$ where $z_2$ is non-zero.}The equation we are dealing here is given by $$\begin{pmatrix}
    a & b\\
    c & d
\end{pmatrix}=\begin{pmatrix}
    \lambda_1 & 0\\
    0 & \lambda_2
\end{pmatrix}x^{k_1}+\begin{pmatrix}
    0 & z_2\\
    1 & 0
\end{pmatrix}y^{k_2}$$Here, $\widetilde B$ is non-singular as $z_2$ is non-zero. If $C$ is non-singular with $J_A$ and $\widetilde B$ both being non-singular, we are done. Consider the case when $J_A$ is singular. We have two cases:
\paragraph{$\lambda_1=0$ and $\lambda_2\neq 0$} For $b\neq 0$, let $y=\begin{pmatrix}
    0 & 0\\
    b_0 & b_1
\end{pmatrix}$. Then $C-\widetilde By^{k_2}$ is given by $\begin{pmatrix}
    a-z_2b_1b_0^{k_2-1} & b-z_2b_0^{k_2}\\
    c & d
\end{pmatrix}$. Choose $b_0$ such that $b-z_2b_0^{k_2}=0$ and $b_1$ such that $a-z_2b_1b_0^{k_2-1}=0$. If $b$ is zero, let $y=\begin{pmatrix}
    1 & 0\\
    b_0 & 0
\end{pmatrix}$. Then $C-\widetilde By^{k_2}$ is given by $\begin{pmatrix}
    a-z_2b_0 & 0\\
    c-1 & d
\end{pmatrix}$. Choose $b_0$ such that $a-z_2b_0 =0$. It suffices to show $\begin{pmatrix}
    0 & 0\\
    c & d
\end{pmatrix}$ is in the image of $J_Ax^{k_1}$, which we have already done in $3.2.7.4$.
\paragraph{$\lambda \neq 0$ and $\lambda_2=0$}If $c\neq 0$, let $y=\begin{pmatrix}
    b_0 & b_1\\
    0 & 0
\end{pmatrix}$. Otherwise, let $y=\begin{pmatrix}
    0 & b_0\\
    0 & 1
\end{pmatrix}$. In either case, we choose $b_0$ and $b_1$ such that $C-\widetilde By^{k_2}$ is given by matrices of the form $\begin{pmatrix}
    a & b\\
    0 & 0
\end{pmatrix}$. Letting $x=\begin{pmatrix}
    a_0 & a_1\\
    0  & 0
\end{pmatrix}$ for $a\neq 0$ and $x=\begin{pmatrix}
    0 & a_0\\
    0 & 1
\end{pmatrix}$ for $a=0$ gives us the desired result.
{The given map is \textcolor{red}{surjective.} \qed
\begin{proposition}\label{lem:power-A-unipotent}
    Let $\omega=Ax^{k_1}+ By^{k_2}\in\M(2,K)\langle x, y\rangle$, with $A$, $B$ nonzero matrices. If up to conjugation $A$ is a unipotent matrix $\begin{pmatrix}
        \lambda & 1 \\ & \lambda
    \end{pmatrix}$, then the map $\widetilde{\omega}$ is surjective if and only if one of the following cases occurs;
    \begin{enumerate}
        \item $A$ is invertible
        \item if $\lambda=0$, the second row of an orbit representative of $B$ 
        has a nonzero second row.
    \end{enumerate}
\end{proposition}
\subsection{Proof of \cref{lem:power-A-unipotent}}In the existing section, we are dealing with the upper triangular Jordan form with the same diagonals. Thus, the equation is given by $$C=\begin{pmatrix}
    \lambda_1 & 1\\
    0 & \lambda_1
\end{pmatrix}x^{k_1}+\widetilde By^{k_2}.$$ 
\begin{lemma}\label{lem:uniptotent-B-nsing}
    Let $\widetilde B$ be a non-singular matrix. For $J_A$ given by $\begin{pmatrix}
        \lambda_1 & 1\\
        0 & \lambda_1
    \end{pmatrix}$ with $\lambda_1\neq 0$, the map $\widetilde{\omega}$ is surjective.
\end{lemma}
\begin{proof}
    Let $x=\begin{pmatrix}
    a_0 & a_1\\
    0 & a_0
\end{pmatrix}$. Then the $\det(C-J_Ax^{k_1})$ is given by $$-\left(a+d\right)\lambda_1a_0^{k_1}+c\left(a_0+\lambda_1k_1a_1\right)a_0^{k_1-1}+\lambda_1^2a_0^{2k_1}.$$ Choose $a_0$ and $a_1$ such that $\det\left(C-J_Ax^{k_1}\right)\neq 0$. Hence, $\widetilde B$ being non-singular gives us the solution.  
\end{proof}
The action of the centralizer of Jordan form considered gives the representatives discussed below:
\subsubsection{For $\mu_1,\mu_2\in k^\times$ with $z\neq 0$, we have $\widetilde B$ given by $\begin{pmatrix}
    \mu_1  & 0\\
    z & \mu_2
\end{pmatrix}$}Here, we take into account both the cases $\mu_1\neq \mu_2$ and $\mu_1=\mu_2$. Clearly, $\widetilde B$ given is a non-singular matrix as $\mu_i\neq 0$. We examine the cases where $J_A$ is singular or non-singular. If $J_A$ is singular, the equation is given by $$\begin{pmatrix}
    a & b\\
    c & d
\end{pmatrix}=\begin{pmatrix}
    0 & 1\\
    0  & 0
\end{pmatrix}x^{k_1}+\begin{pmatrix}
    \mu_1 & 0\\
    z  & \mu_2
\end{pmatrix}y^{k_2}.$$ If $c$ or $d$ is non-zero, let $x=\begin{pmatrix}
    0 & 0\\
    a_1 & a_0
\end{pmatrix}$. By choosing appropriate $a_0$ and $a_1$, $C-J_Ax^{k_1}$ is a non-singular matrix and $\widetilde B$ being non-singular gives us the required $y$. If $c$ and $d$ both are zero, then choose the same $x$ considered above with $b\neq 0$. Then $C-J_Ax^{k_1}$ is $\begin{pmatrix}
    a-a_1a_0^{k_1-1} & b-a_0^{k_1}\\
    0 & 0
\end{pmatrix}$. Choose $a_0$ such that $b-a_0^{k_1}=0$ and hence choose $a_1$ such that $a-a_1a_0^{k_1-1}=0$ and let $y=0$. If $b$ is zero, then let $y=0$ and $x=\begin{pmatrix}
    1 & 0\\
   a & 0
\end{pmatrix}$. Hence, for $J_A$ being singular, the map is \textcolor{red}{surjective}. Consider the case where $J_A$ is non-singular. The map is \textcolor{red}{surjective} by \cref{lem:uniptotent-B-nsing}.
\subsubsection{For $\mu_1\neq \mu_2$, $\widetilde B$ is given by $\begin{pmatrix}
    \mu_1 & 0\\
    0 & \mu_2
\end{pmatrix}$} The equation being dealt in this case is given by $$\begin{pmatrix}
    a & b\\
    c & d
\end{pmatrix}=\begin{pmatrix}
    \lambda_1 & 1\\
    0  & \lambda_1
\end{pmatrix}x^{k_1}+\begin{pmatrix}
    \mu_1 & 0\\
    0 & \mu_2
\end{pmatrix}y^{k_2}.$$ The proof follows from $3.2.3$.
\subsubsection{For $\mu_1,z\in k$, the representative is given by $\begin{pmatrix}
    \mu_1 & z\\
    0 & \mu_1
\end{pmatrix}$}We start with $C$ being singular and $J_A$ and $\widetilde B$ being non-singular. The 
map is \textcolor{red}{surjective} by lemma $3.1$. If $J_A$ is singular and $\widetilde B$ is non-
singular, the proof is similar to $3.3.1$ as we only use the fact that $\widetilde B$ is non-singular. 
For $J_A$ being non-singular and $\widetilde B$ being singular, we let $y$ to be the $x$ considered in 
$3.3.1$ and use the fact that $J_A$ is non-singular. If $J_A$ and $\widetilde B$ both are singular, 
the map is \textcolor{red}{not surjective} as the matrices $C$, having non-zero $c$ or $d$ are not in the image. 
\section{Proof of theorem B}\label{sec:proof-theorem-B}
In this section we consider $\omega\in \M(2,k)\langle x, y \rangle$ to be $Axy-Byx$,
where $A,B$ are nonzero matrices.
This an analogue of the famous L\'{v}ov-Kaplansky conjecture, which has 
become point of research in recent times. The L\'{v}ov-Kaplansky conjecture  states that if $K$ is an infinite field, $\A=\M(n,K)$ is the algebra of $n\times n$ matrices over $K$, then the image of a multilinear polynomial $p$
with coefficient in $K$, must be one of the following four vector spaces: $\{0\}$, the space of scalar matrices, the space of traceless matrices, and $\A$. 
We consider the polynomial which is similar to a commutator but with 
matrix coefficients. It is known that in case the polynomial is $xy-yx$, then
the image of $\M(n,K)$ is set of all trace zero matrices, and in particular a 
vector space. In this section, 
we prove that the image of $Axy-Byx$ is a vector space.
When $A=B$, the image of this polynomial is given by $A\mathfrak{sl}(2,k)$ and hence it is a vector 
subspace. We determine the images of this polynomial when $A\neq B$.
Furthermore, if $A-B$ is invertible, then for a matrix 
$M\in\M(2,K)$, choosing $x=(A-B)^{-1}M$ and $Y$ to be the $2\times 2$ identity matrix, we get that
the image of this polynomial is the whole of $\M(2,K)$. Thus in the following discussion, we assume that 
$\det(A-B)=0$. As was mentioned in \cref{sec:reduction}, we use 
the simultaneous conjugacy classes and solve the problem 
on the reduced equation. By abuse of notation, we continue using $A$, $B$ in place of $J_A$ and $\widetilde{B}$ respectively.
\subsection{Suppose $A$ is a scalar matrix.} We will divide this into two different cases, depending on the orbit representatives.
\subsubsection{Take $A=\begin{pmatrix}
    \mu_1&\\&\mu_1
\end{pmatrix}$, and ${B}=\begin{pmatrix}
    \mu_1\\&\mu_2
\end{pmatrix}$.} Then 
\begin{align*}
    &\begin{pmatrix}
    \mu_1&\\&\mu_1
\end{pmatrix}\begin{pmatrix}
    a_1 &  \\
     a_3 & 1
\end{pmatrix}\begin{pmatrix}
     & 1\\
    b_3 & b_4
\end{pmatrix}
-
\begin{pmatrix}
    \mu_1\\&\mu_2
\end{pmatrix}\begin{pmatrix}
     & 1\\
    b_3 & b_4
\end{pmatrix}\begin{pmatrix}
    a_1 &    \\
       a_3 & 1
\end{pmatrix}\\
&=\begin{pmatrix}
    -\mu_1a_3 & \mu_1(a_1-1)\\
    \mu_1b_3-\mu_2a_1b_3-\mu_2a_3b_4 & \mu_1a_3+(\mu_1-\mu_2)b_4
\end{pmatrix}.
\end{align*}
Since $\mu_1-\mu_2\neq 0$ and $\mu_1\neq 0$, if we choose $a_3$, $a_1$, $b_4$ and $b_3$ in the given order, it will cover all matrices of $\M(2,K)$. In this case the map is surjective and hence a vector space.
\subsubsection{Take $A=\begin{pmatrix}
    \mu&\\&\mu
\end{pmatrix}$, and ${B}=\begin{pmatrix}
    \mu&1\\&\mu
\end{pmatrix}$.}Then 
\begin{align*}
    &\begin{pmatrix}
    \mu&\\&\mu
\end{pmatrix}\begin{pmatrix}
    a_1 & a_2 \\
      & a_4
\end{pmatrix}\begin{pmatrix}
    b_1 & b_2\\
    1 & 
\end{pmatrix}
-
\begin{pmatrix}
    \mu&1\\&\mu
\end{pmatrix}\begin{pmatrix}
    b_1 & b_2\\
    1 & 
\end{pmatrix}\begin{pmatrix}
    a_1 &  a_2  \\
        & a_4
\end{pmatrix}\\
&=\begin{pmatrix}
    \mu a_2-a_1 & \mu a_1b_2-\mu a_2b_1-a_2-\mu a_4b_2\\
    \mu a_4 -\mu a_1 & -\mu a_2
\end{pmatrix}.
\end{align*}
Since $\mu\neq 0$, choosing the elements in the order $a_2$, $a_1$, $a_4$ and $b_2$ or $b_1$ (when at 
least one of $a_1,$ $a_2$, $a_4$, $a_1-a_4$ is nonzero), we get all matrices but $\begin{pmatrix}
  0  & \alpha \\ &0
\end{pmatrix}$. But this matrix can be obtained by setting $X$ to be the identity matrix and $Y=\diag(\alpha,\alpha)$. In this case the map gives the image to be $M(2,K)$ and hence is a vector space.
\subsection{Now assume $A$ to be a diagonal matrix} We fix $A$ to be the matrix $\begin{pmatrix}
    \mu_1 & \\ & \mu_2
\end{pmatrix}$ with $\mu_1\neq \mu_2$. We will divide this case into several subcases. Just keep in mind that 
we are assuming $\det(A-B)=0$. A few subcases have been discussed previously. So we won't repeat them.
\subsubsection{Take ${B}=\begin{pmatrix}
    \mu_1 & 1 \\ & \mu_1
\end{pmatrix}$} In this case,
\begin{align*}
    &\begin{pmatrix}
    \mu_1&\\&\mu_2
\end{pmatrix}\begin{pmatrix}
    a_1 & a_2 \\
      & a_4
\end{pmatrix}\begin{pmatrix}
    b_1 & b_2\\
    1 & 
\end{pmatrix}
-
\begin{pmatrix}
    \mu_1&1\\&\mu_1
\end{pmatrix}\begin{pmatrix}
    b_1 & b_2\\
    1 & 
\end{pmatrix}\begin{pmatrix}
    a_1 &  a_2  \\
        & a_4
\end{pmatrix}\\
&=\begin{pmatrix}
    \mu_1a_2-a_1 & \mu_1a_1b_2-\mu_1a_2b_1-a_2-\mu_1a_4b_2\\
    \mu_2a_4-\mu_1a_1 & \mu_1a_2
\end{pmatrix}.
\end{align*}
If $\mu_1\mu_2\neq 0$, the above expression gives all the matrices of $\M(2,K)$ but the matrices of the form
$\begin{pmatrix}
 0 & \alpha \\ & 0   
\end{pmatrix}$. But this matrix can be obtained by plugging $X=\begin{pmatrix}
    0 & 1\\ 0& 0
\end{pmatrix}$, and $y=\begin{pmatrix}
    0 & 0 \\0 & \alpha/\mu_1
\end{pmatrix}$. Hence assume $\mu_1\mu_2=0$ hereafter, but not both being zero. In case $\mu_1=0$, the equation
\begin{align*}
    &\begin{pmatrix}
        0 & \\
         & \mu_2
    \end{pmatrix}
    \begin{pmatrix}
        a & b\\
        1 & 
    \end{pmatrix}
    \begin{pmatrix}
        e & f\\
        1 & 
    \end{pmatrix}
    -
    \begin{pmatrix}
       0  & 1\\
         & 0
    \end{pmatrix}
    \begin{pmatrix}
        e & f\\
        1 & 
    \end{pmatrix}
    \begin{pmatrix}
        a & b\\
        1 & 
    \end{pmatrix}
    =\begin{pmatrix}
        - a & -b\\
        \mu_2e & \mu_2f
    \end{pmatrix},
\end{align*}
proves that the map is surjective and hence the {image is a vector space}. Moving to the case
$\mu_2=0$, we have 
\begin{align*}
    &\begin{pmatrix}
        \mu_1 & \\
         & 0
    \end{pmatrix}
    \begin{pmatrix}
        0 & 1\\
        1 & 0
    \end{pmatrix}
    \begin{pmatrix}
        e & f\\
        g & h
    \end{pmatrix}
    -
    \begin{pmatrix}
       \mu_1  & 1\\
         & \mu_1
    \end{pmatrix}
    \begin{pmatrix}
        e & f\\
        g & h
    \end{pmatrix}
    \begin{pmatrix}
        0 & 1\\
        1 & 0
    \end{pmatrix}
    =\begin{pmatrix}
        \mu_1g-\mu_1f-h & \mu_1h-\mu_1e-g\\
        -\mu_1h & -\mu_1g
    \end{pmatrix},
\end{align*}
which shows that if we choose $g$, $h$, $f$ and $e$ one after another, we get all of $\M(2,K)$. This 
proves that the map is surjective. A little bit of modification of the equation also gives 
the surjectivity of the map when ${B}$ is of the form 
$\begin{pmatrix}
        \mu_1 & \\
        1 & \mu_1
    \end{pmatrix}$, $\begin{pmatrix}
        \mu_2 & 1\\
         & \mu_2
    \end{pmatrix}$, and $\begin{pmatrix}
        \mu_2 & \\
        1 & \mu_2
    \end{pmatrix},$. We move to the next case.
\subsubsection{Take ${B}=\begin{pmatrix}
    \mu_1 & 1\\
    0 & \mu_2
\end{pmatrix}$} First, let us assume that $\mu_1\mu_2\neq 0$. Observe that
\begin{align*}
    &\begin{pmatrix}
        \mu_1 & \\ & \mu_2
    \end{pmatrix}
    \begin{pmatrix}
        a & b \\ 1 & d
    \end{pmatrix}
    \begin{pmatrix}
        e & f \\ 0 & h
    \end{pmatrix}
    -
    \begin{pmatrix}
        \mu_1 & 1 \\ &\mu_2
    \end{pmatrix}
    \begin{pmatrix}
        e & f \\ 0 & h
    \end{pmatrix}
    \begin{pmatrix}
        a & b \\ 1 & d
    \end{pmatrix}
    \\=&\begin{pmatrix}
        -\mu_1f-h & \mu_1af+\mu_1bh-\mu_1be-\mu_1fd-dh\\
        \mu_2e-\mu_2h & \mu_2f
    \end{pmatrix}.
\end{align*}
Hence choosing $f$, $h$, $e$ in order, and other variables appropriately, attains all the matrices of 
$\M(2,K)\setminus\{U_{\alpha}:\alpha\neq 0\}$, where $U_{\alpha}=\begin{pmatrix}
        0 & \alpha \\& 0
    \end{pmatrix}$. Indeed if $e=f=h=0$ then the resultant matrix is zero. But, we can get the
matrices $U_\alpha$ by plugging $X=\begin{pmatrix}
    0 & \\
    & \alpha
\end{pmatrix}$, and $Y$ to be the $2\times 2 $ identity matrix. 

Now, suppose $\mu_2=0$. Then plugging $X=\begin{pmatrix}
    &\\ c & d
\end{pmatrix}$ and $Y$ to be the identity matrix, we get that the image consists of all matrices of the form 
$\begin{pmatrix}
    a & b \\ 0 & 0
\end{pmatrix}$, which is a vector space. The case $\mu_1=0$ and $\mu_2\neq 0$ is similar. Hence, in this
subcase, the image is a vector space. A similar argument proves the case $B=\begin{pmatrix}
    \mu_1 & \\ 1 & \mu_2
\end{pmatrix}$.
\subsubsection{Consider $B=\begin{pmatrix}
    & z_2 \\ 1 &
\end{pmatrix}$} Since $\det(A-B)=0$, we conclude that $z_2=\mu_1\mu_2$. Since 
$z_2\neq 0$, we get that $\mu_1\mu_2\neq 0$, forcing each of $\mu_i$'s to be nonzero. For $\mathrm{char}(k)\neq 2$, we divide this into two subcases, depending on the 
value of $\mu_1$. For $\mu_1\neq 1/2$, choosing 
\begin{align*}
    x=\begin{pmatrix}
    a & b\\ c& d
\end{pmatrix}\text{ and }y=\begin{pmatrix}
    \mu_1 &-\mu_1\mu_2\\
     & 1
\end{pmatrix},
\end{align*} we get the image to be $\M(2,K)$ which is a vector space. For the 
other case we get the full of $\M(2,K)$ by setting
\begin{align*}
    x=\begin{pmatrix}
        2 & -2\mu_2\\ &1
    \end{pmatrix}\text{ and }y=\begin{pmatrix}
        e & f \\ g& h
    \end{pmatrix}.
\end{align*}
The $\mathrm{char}(k)=2$ can be proved by a choice of same kind of matrices.
\subsubsection{Consider $B=\begin{pmatrix}
     & z_2 \\1 & z_3
\end{pmatrix}$, with $z_3\neq 0$} First we deal with the case $z_2=0$. Then
\begin{align*}
    &\begin{pmatrix}
        \mu_1 & \\ & \mu_2
    \end{pmatrix}
    \begin{pmatrix}
         & 1\\ 1 &  
    \end{pmatrix}
    \begin{pmatrix}
        e & f \\ g& h
    \end{pmatrix}- \begin{pmatrix}
         & 0\\ 1 & z_3
    \end{pmatrix}
    \begin{pmatrix}
        e & f \\ g& h
    \end{pmatrix}
    \begin{pmatrix}
        & 1\\
        1 &
    \end{pmatrix}\\
    =&\begin{pmatrix}
        \mu_1g & \mu_1h \\ \mu_2 e-z_2h  & \mu_2 f-z_2 g
    \end{pmatrix}.
\end{align*}
Since we have four linear equation in four variables and the coefficient matrix of
the equations have non-zero determinant, we get that the image is set of all matrices of the form 
$\begin{pmatrix}
    & \\ x & y
\end{pmatrix}$, if $\mu_1=0$, or the whole of $\M(2,K)$ otherwise. This proves that
the image is a vector space. For the other case choose $\lambda, \lambda'\neq 0$ such that $\mu_1^2\neq \lambda$ and $\mu_1^2\lambda\neq z_2^2\lambda'$ and 
$\mu_2^2\lambda'\neq \lambda$. Then choosing 
\begin{align*}
    X=\begin{pmatrix}
        & \lambda \\ \lambda'& 
    \end{pmatrix}, \text{ and }Y=\begin{pmatrix}
        e & f \\ g & h
    \end{pmatrix},
\end{align*}
we get the image of the polynomial to be vector space as it is $\M(2,K)$. A similar argument works for the case $B=\begin{pmatrix}
    z_1 & z_2 \\ 1 &
\end{pmatrix}$. 
\subsubsection{Lastly, assume $B=\begin{pmatrix}
    z_1 & z_2 \\ 1 & z_3
\end{pmatrix}$, with $z_i\neq 0$} Since $\det (A-B)= 0$, we get that $z_2=(\mu_1-z_1)(\mu_2-z_3)$. Furthermore, $z_2\neq 0$ implies that $\mu_1\neq z_1$
and $\mu_1\mu_2\neq z_2$. First we assume that $\mu_1\neq 0$. Then we have that
\begin{align*}
    &\begin{pmatrix}
        \mu_1 & \\ & \mu_2
    \end{pmatrix}
    \begin{pmatrix}
        a & b \\ c & d
    \end{pmatrix}
    \begin{pmatrix}
        1 & \\ & f
    \end{pmatrix}
    -\begin{pmatrix}
        z_1 & z_2 \\ 1 & z_3
    \end{pmatrix}
    \begin{pmatrix}
        1 & \\ & f
    \end{pmatrix}
    \begin{pmatrix}
        a & b \\ c& d
    \end{pmatrix}\\
    &=\begin{pmatrix}
        (\mu_1-z_1)a-z_2fc &  b(\mu_1f-z_1)-z_2fd\\
        (\mu_2-z_3f)c-a & d(\mu_2-z_3)f-b
    \end{pmatrix}.
\end{align*}
Note that the coefficient matrix for $a,b,c,d$ is 
\begin{align*}
    \begin{pmatrix}
        \mu_1-z_1 & & -z_2f & \\
        & \mu_1f-z_1 && -z_2f\\
        -1 & & \mu_2-z_3f & \\
        & -1 & & (\mu_2-z_3)f
    \end{pmatrix}.
\end{align*}
Then the determinant of the coefficient matrix is $-\mu_1\mu_2z_2f(f-1)^2$. 
Then choosing $f=2$, we get the image of the map to be whole of $\M(2,K)$ and 
hence it is
a vector space. We do the case $\mu_1=0$ and leave the other case for the reader as the trick is similar.
We have
\begin{align*}
    &\begin{pmatrix}
        0& \\ & \mu_2
    \end{pmatrix}
    \begin{pmatrix}
        a & b \\ c & d 
    \end{pmatrix}
    \begin{pmatrix}
        & 1\\  \lambda & 
    \end{pmatrix}
    -
    \begin{pmatrix}
        z_1 & z_2 \\ 1 & z_3
    \end{pmatrix}
    \begin{pmatrix}
        & 1\\ \lambda &
    \end{pmatrix}
    \begin{pmatrix}
        a & b \\ c & d
    \end{pmatrix}\\
    =&\begin{pmatrix}
        -z_2\lambda - z_1c & -z_2\lambda b - z_1d\\
        \mu_2 d\lambda -\lambda z_3a - c & \mu_2 c - \lambda z_3b-d
    \end{pmatrix}.
\end{align*}
Choose $\lambda$ such that $\lambda z_2^2(1-\lambda\mu_2^2)+\lambda^2z_1^2z_3^2\neq 0$, which 
is the coefficient matrix of the equations involving $a$, $b$, $c$, and $d$. 
This proves that in this case the image is surjective on $\M(2,K)$. The case when
$\mu_2=0$ and $\mu_1\neq 0$, is very similar to this case. This finishes all the subcases, proving it to be a vector space.

\subsection{Now we conclude the final case when $A$ is a unipotent matrix} We fix $A$ to be the matrix
$\begin{pmatrix}
    \mu & 1 \\
    & \mu
\end{pmatrix}$. As in all other cases, this will be divided into a few subcases. Also, we should remember our standing assumption that $\det(A-B)=0$. The previously covered cases will not be repeated.
\subsubsection{Take $B=\begin{pmatrix}
    \mu_1 & \\
    z & \mu_2
\end{pmatrix}$, with $z\neq 0$} Since $\det(A-B)=0$, we get that 
$\lambda\neq \mu_2$.
In this case
\begin{align*}
    &\begin{pmatrix}
        \lambda & 1 \\ & \lambda
    \end{pmatrix}
    \begin{pmatrix}
        a & b\\ c & d
    \end{pmatrix}
    \begin{pmatrix}
        1 & \\ 1 & h
    \end{pmatrix}
    -
    \begin{pmatrix}
        \mu_1 & \\ z & \mu_2 
    \end{pmatrix}
    \begin{pmatrix}
        1 & \\ 1 & h 
    \end{pmatrix}
    \begin{pmatrix}
        a & b \\ c & d
    \end{pmatrix}\\
    =&\begin{pmatrix}
        (\lambda-\mu_1)a +\lambda b + c + d & (\lambda h -\mu_1)b + d\\
        -(z+\mu_2)a +(\lambda-\mu_2h)c+\lambda d & -(z+\mu_2 )b+ 
        h(\lambda_1-\mu_2)d
    \end{pmatrix}.
\end{align*}
Since $\lambda$ and $\mu_2$ can not be simultaneously zero and the determinant of the coefficient matrix of $a,b,c,d$ is 
$h[(\lambda-\mu_2)\lambda(h-1)+\mu_2][(\lambda-\mu_1)\mu_2(1-h)+\mu_2g]$, we can choose $h$ such a way that
the determinant is nonzero. Hence we get the image is $\M(2,K)$ which is a vector space.
\subsubsection{Consider $B=\begin{pmatrix}
    \lambda & z \\ & \lambda
\end{pmatrix}$} Firstly assume that $z\neq 0$. In this case
\begin{align*}
&\begin{pmatrix}
    \lambda & 1 \\
    & \lambda
\end{pmatrix}
\begin{pmatrix}
   a & b \\ c & d
\end{pmatrix}
\begin{pmatrix}
     & f \\
    1 &  
\end{pmatrix}
-
\begin{pmatrix}
    \lambda & z\\
     & \lambda
\end{pmatrix}
\begin{pmatrix}
     & f \\
    1 & 
\end{pmatrix}
\begin{pmatrix}
   a & b \\ c & d
\end{pmatrix}\\
=&\begin{pmatrix}
    \lambda b + d - \lambda f c - za & \lambda a f + c f -\lambda f d - zb\\
    \lambda d -\lambda a & \lambda c f -\lambda b
\end{pmatrix}.
\end{align*}
Then $f$ can be chosen such a way that the coefficinet matrix of the system of equations in $q,b,c,d$ has non-zero determinant. This shows 
that the image in this case if $\M(2,K)$.

The case when $z=0$ is easier, and it can be shown that the image of this polynomial is given by the set of matrices of the form $\begin{pmatrix}
    a & b \\ &
\end{pmatrix}$, which is again a vector space. This finishes the proof of \cref{thm:B}.

\section{Concluding remarks}\label{sec:concl-rem}
In summary, we have proved that the generalized Waring-type problem has a positive solution
for $\M(2,K)$ under mild conditions on the matrices. Furthermore, the
image of $Axy-Byx$ for nonzero $A,B\in\M(2,K)$ is a vector space for all choices of 
$A$, $B$. It will be interesting to see images of other polynomials in general 
$\M(n,K)$. Generalizing \cref{thm:A} we propose the following problem.
\begin{problem}
    Let $K$ be an algebraically closed field. Let $A$ and $B$ be two nonzero 
    matrices in $\M(n,K)$. Then the map induced by $Ax^{k_1}+By^{k_2}$ with 
    $k_1,k_2\geq 2$, on $\M(n,K)$ is surjective if and only if 
    under simultaneous conjugation, there exists representative of the orbit space say $A'$ and $B'$ such that $A'$ and $B'$ have disjoint places 
    of zero rows.
\end{problem}
A polynomial $F\in \A\langle x_1, x_2,\ldots, x_n\rangle$ will be called a 
multilinear polynomial if it is linear in each variable. For example
\begin{align*}
    x_1A_1x_2A_2x_3+x_2A_4x_1x_3A_5-x_3x_1A_6x_2,
\end{align*}
where $A_i\in\M(n,K)$, is a multilinear polynomial in $3$ variables in algebra $\M(n,K)$. Generalizing \cref{thm:B}, we present a generalized version of the 
L\'{v}ov-Kaplansky conjecture.
\begin{problem}
    Let $K$ be a field. For any multilinear polynomial in $\omega\in \M(n,K)\langle x_1,x_2,\ldots, x_m\rangle$, the image $\widetilde{\omega}\left(\M(n,K)\right)$ is a vector space.
\end{problem}
\printbibliography
\end{document}